%% file: Shifted_double_Lie-Rinehart_algebras.tex
\documentclass[10pt]{amsart}

\usepackage{varioref}
\usepackage[french,english]{babel}
\usepackage[utf8]{inputenc}	
\usepackage[T1]{fontenc}

\usepackage[normalem]{ulem}
\usepackage{latexsym}

\usepackage{amsmath}			
\usepackage{amsfonts}			
\usepackage{amssymb}			
\usepackage{amsthm}
\usepackage{thmtools}			
\usepackage{amsbsy}				
\usepackage{wasysym}
\usepackage{dsfont}			
\usepackage{shuffle}		
\usepackage{mathrsfs}

\usepackage{graphicx}
\usepackage{float}	
\usepackage{comment}
\usepackage{layout}
\usepackage{color}			

\usepackage[all]{xy}			
\usepackage{tikz}				
\usetikzlibrary{calc,shapes,shapes.geometric,decorations.pathreplacing,arrows} 

\usepackage{array}				
\newcolumntype{P}[1]{>{\centering\arraybackslash}p{#1}} 

\usepackage{subfig}

\input{commandes}

\usepackage{enumitem}		

\numberwithin{equation}{section}
\allowdisplaybreaks[4]	
\usepackage[colorlinks=true,linkcolor=blue,citecolor=blue,bookmarks=true,bookmarksnumbered=true,bookmarksopenlevel=3]{hyperref}
\title{Shifted double Lie-Rinehart algebras}
\address{Laboratoire angevin de recherches en mathématiques (Larema), Universit\'e d'Angers, COMUE Bretagne Loire, 2 Bd Lavoisier 49045 Angers Cedex 01}
\email{leray@math.univ-angers.fr}
\author[J. Leray]{Johan \textsc{Leray}} 
\date{\today}
\thanks{This article is a part of PhD thesis \cite{Ler17} of the author, who's supported by the project "Nouvelle \'Equipe Topologie alg\'ebrique et Physique Math\'ematique", convention n$^\circ$2013-10203/10204 between La R\'egion des Pays de la Loire and the University of Angers. The author is indebted to G. Powell who has carefully read and corrected the first version of this paper; he also thanks the Centre Henri Lebesgue ANR-11-LABX-0020-01 for its stimulating mathematical research programs. }

\begin{document}

\begin{abstract}
We generalize the notions of shifted double Poisson and shifted double Lie-Rinehart structures, defined by Van den Bergh in \cite{VdB08,VdB08-2}, to monoids in a symmetric monoidal abelian category. The main result is that an $n$-shifted double Lie-Rinehart structure on a pair $(A,M)$ is equivalent to a non-shifted double Lie-Rinehart structure on the pair $(A,M[-n])$.
\end{abstract}

\maketitle
\tableofcontents
\section{Introduction}

\subsection{Noncommutative geometry}
In algebraic geometry, a commutative algebra $C$ over a field $k$ corresponds to an affine scheme $\mathrm{Spec}(C)$, via the  functor of points. The scheme $\mathrm{Spec}(C)$ is the geometrical object associated to its  algebra of functions $C$. Working in noncommutative geometry, a natural question arises: for a noncommutative algebra $A$, which is viewed as an algebra of noncommutative functions, what is the geometrical object associated to $A$? Kontsevich and Rosenberg propose to consider the family of schemes $\big\{\mathrm{Rep}_V(A)/\!/\GL(V)\big\}_V$, \emph{the moduli space of representations of $A$},  as successive approximations of a hypothetical noncommutative affine scheme "$\mathrm{NCSpec}(A)$".

The scheme $\mathrm{Rep}_V(A)$, called  \emph{the scheme of representations} of $A$ on the finite dimensional vector space $V$, is defined by its functor of points:
\[
\begin{array}{rccc}
\mathrm{Rep}_V(A) : & \mathtt{CommAlg}_k & \longrightarrow & \mathtt{Sets} \\
& B & \longmapsto & \Hom_{\mathtt{Alg}_k}(A, \End(V) \otimes B)
\end{array}.
\]
This scheme is affine, i.e. there exists a commutative $k$-algebra, denoted by $A_V$, such that $\mathrm{Rep}_V(A) = \mathrm{Spec}(A_V)$.  The algebra $A_V$ can be given by generators and relations: it is spanned by the set $\big\{ a_{ij} ~\big|~ a\in A, 1\leqslant i,j\leqslant \mathrm{dim}_k V \big\}$ for a fixed isomorphism $k^{\mathrm{dim}_k V }\cong V$, with the relations  $\alpha a_{ij} = (\alpha a)_{ij}$, $a_{ij}+b_{ij}=(a+b)_{ij}$, $\sum_{l} a_{il}b_{lj}=(ab)_{ij}$ and $ 1_{ij}= \delta_{ij}$ for all $\alpha\in k$, $a,b\in A$, where $\delta$ is Kronecker's symbol. The quotient of the representation scheme $\mathrm{Rep}_V(A)/\!/\GL(V)$ corresponds to taking $A^{\GL(V)}$, the invariant part of the algebra $A_V$ for the action by conjugation of $\GL(V)$.

Kontsevich and Rosenberg assert that any (noncommutative) property of "$\mathrm{NCSpec}(A)$" should induce its commutative analogue on $\mathrm{Rep}_V(A)/\!/\GL(V)$, for all $V$: this is the \emph{Kontsevich-Rosenberg principle}. Following this principle, many authors have developed noncommutative structures; the reader can refer to \cite{Gin05} for such constructions in noncommutative geometry.

\subsection{Noncommutative Poisson brackets}
It is natural to ask what is the good definition of a noncommutative Poisson structure. Recall that a Poisson bracket on an associative commutative $k$-algebra $B$  is a Lie bracket $\{-,-\} : B\otimes B\rightarrow  B$ which satisfies the Leibniz rule $\{ab,c\}=a\{b,c\}+\{a,c\}b$ for all $a,b$ and $c$ in $B$. For noncommutative algebras, this definition is too restrictive, as show by \cite[Th. 1.2]{FL98}: for $A$, an associative algebra with a noncommutative domain, i.e. $[A,A]\ne 0$, a Poisson bracket is the commutator, up to a multiplicative constant. In \cite{C-B11}, Crawley-Boevey gives us the minimal structure on an associative algebra $A$ which induces a Poisson bracket on  $A_V^{\GL(V)}$, which he calls an  $\mathrm{H}_0$-Poisson structure. An $\mathrm{H}_0$-Poisson structure on $A$ is a Lie bracket on  $A_\natural=A/[A,A]$   such that, for all $a\in A$  (with class $\bar{a}$ in $A_\natural$), the application $\langle \bar{a}, -\rangle : A_\natural \rightarrow A_\natural$ is induced by a derivation $d_a :A \rightarrow A$. Crawley-Boevey shows that, if $A$ is a $\mathrm{H}_0$-Poisson algebra, then there exists a unique Poisson structure on $A_V^{\GL(V)}$ that is compatible with the trace morphism (see \cite[Th. 1.6]{C-B11}).

An important remark is the following: there are few examples of $\mathrm{H}_0$-Poisson structures which do not arise from a richer structure. A good example of such a structure is a double Poisson bracket, defined by Van den Bergh in \cite{VdB08}. A double Poisson bracket on an associative algebra $A$ is a morphism
\[
	\begin{array}{rccc}
	\DB{-}{-} : & A\otimes A &\longrightarrow &A\otimes A \\
	& a\otimes b & \longmapsto & \DB{a}{b}'\otimes \DB{a}{b}''
	\end{array}
\]
(using Sweedler's notations) which is antisymmetric, i.e. for all elements $a$ and $b$ in $A$, $\DB{a}{b}=-\DB{b}{a}''\otimes \DB{b}{a}'$, which is a derivation in its second variable and satisfies  \emph{the double Jacobi relation} (see definition \ref{def::doubleJacobi}). There are lots of examples.  In \cite{VdW08}, Van de Weyer studies double Poisson brackets on semi-simple algebras of finite dimension. This structure is adapted to the noncommutative world: for example, in \cite{Pow16}, Powell shows that any double Poisson bracket on a free commutative algebra with at least two generators is trivial. Van den Bergh shows that (see \cite[Lem. 2.6.2]{VdB08}) a double Poisson bracket $\big( A,\DB{-}{-}\big)$ induces a $\mathrm{H}_0$-Poisson structure on $A$, where the Lie bracket $\{-,-\}_\natural$ is induced by $\{-,-\}_\natural:=\mu\circ\DB{-}{-}$, with $\mu$ the associative product on $A$. Double Poisson brackets are connected with many mathematical areas, as we'll now see.

In symplectic geometry, we can associate to an exact symplectic manifold $M$ its Fukaya category $\mathrm{Fuk}(M)$ (see \cite{Aur14}). For an exact symplectic $2d$-dimensional manifold, with vanishing first Chern class,  Chen \emph{et al.} show in \cite[Th.17]{CHSY15} that the linear dual of the reduced bar construction of $\mathrm{Fuk}(M)$ has a naturally defined  $(2-d)$-double Poisson bracket. This implies that the cyclic cohomology $\mathrm{HC}^\bullet(\mathrm{Fuk}(M))$ has a $(2-d)$-Lie bracket, an analogue of the Chas-Sullivan bracket in string topology (see \cite[Cor. 19]{CHSY15}).

To a finite quiver $Q$ (see \cite{CBEG05}), Van den Bergh associates the algebra $k\bar{Q}$ of the double quiver $\bar{Q}$,  which has a natural double Poisson bracket (see \cite{VdB08}) which induces the Kontsevich Lie bracket on $(k\bar{Q})_\natural$.

Other examples are related to loop spaces of manifolds with boundary (see \cite{MT13}, \cite{MT14}), the Kashiwara-Vergne problem (see \cite{AKKN17}), and noncommutative integrable systems (see \cite{dSKV15,dSKV16,Art15,Art17}).

\subsection{In this article} This paper is in two parts. In the first, we extend the definition of a shifted double Poisson algebra to monoids in an additive symmetric monoidal category $(\C,\otimes)$. For $\Sigma$ an element of the Picard group of $\C$ and $A$ an associative monoid in $\C$, a $\Sigma$-double Poisson bracket on $A$ is a morphism
\[
\DB{-}{-} : \Sigma A\otimes \Sigma A \longrightarrow \Sigma A\otimes A
\]
where $\Sigma A:= \Sigma \otimes A$, which satisfies antisymmetry and derivation properties (see definition \ref{def::dcrochet_decale}) and the double Jacobi identity (see definition \ref{def::doubleJacobi}). 

In the second part, we study a particular type of double Poisson algebras called \emph{linear double Poisson algebras}. They correspond to double Lie-Rinehart algebras (called double Lie algebroids by Van den Bergh in \cite[Sect. 3.2]{VdB08}), which are the noncommutative version of Lie-Rinehart algebras (see \cite{ACF15,Alv15}). The principal result of this paper is the shifting property of double Lie Rinehart algebras.
\begin{theorem*}[{cf. Theorem \ref{prop::equivalence_decalage}}]
	Let $\C $ be an additive symmetric monoidal category $\C$, with unit $\mathds{1}$, a monoid $A$, an $A$-bimodule $M$ and $\Sigma$ an invertible object in $\C$. The following assertions are equivalents: 
	\begin{enumerate}
		\item $(A,M,\rho_M,\DB{-}{-}_M)$ is a $\Sigma$-double Lie Rinehart algebra;
		\item $(A,\Sigma M,\rho_{\Sigma M},\DB{-}{-}_{\Sigma M})$ is a $\mathds{1}$-double Lie-Rinehart algebra.
	\end{enumerate}
	We have the equivalence of categories
	\[
	\Sigma\mbox{-}\mathtt{DLR}_A \cong \mathds{1}\mbox{-}\mathtt{DLR}_A,
	\]
	with $\Sigma$-$\mathtt{DRL}_A$, the category of $\Sigma$-double Lie-Rinehart algebras over the associative algebra $A$.
\end{theorem*}
An example of such an algebra is given by Van den Bergh in \cite[App. A]{VdB08}: the Koszul double bracket. We extend this example in the general case of a monoid in an additive symmetric monoidal category without shifting (cf. section \ref{subsect::dcrochet_Koszul}):
\begin{theorem*}[cf. Theorem \ref{thm::koszul}]
	Let $A$ be a $\Sigma$-double Poisson algebra in an additive symmetric monoidal category  $(\C,\otimes,\tau)$. The free  $A$-algebra $T_A\Omega_A$ is a linear $\Sigma$-double Poisson algebra.
\end{theorem*}

\section{Notation and algebraic background}
\subsection{Symmetric monoidal category}
	We consider $(\mathtt{C},\otimes)$, an additive category with monoidal structure $\otimes$,  unit $\mathds{1}$ and such that the bifunctor $-\otimes - : \mathtt{C}\times \mathtt{C} \rightarrow \mathtt{C}$ is additive in each entry. We assume that $(\mathtt{C},\otimes)$ is symmetric for the natural transformation $\tau$: for all objects $A$ and $B$ in $\mathtt{C}$, we have the isomorphism $\tau_{A,B}:A\otimes B \rightarrow B\otimes A$ which satisfies $\tau_{B,A}\tau_{A,B}=\mathrm{id}_{A\otimes B}$. 
	
	Throughout this paper, we fix $\Sigma$, an invertible object in $\C$, with inverse $\Sigma^-$ and an isomorphism $\rho : \Sigma^{-}\otimes \Sigma \rightarrow \mathds{1}$. By the symmetry of $\C$, we also have the isomorphism  $\rho\tau_{\Sigma,\Sigma^-} : \Sigma\otimes \Sigma^{-} \cong \mathds{1}$. In particular, $\Sigma$ induces the functor $\Sigma \otimes - : \mathtt{C} \rightarrow \mathtt{C}$ which is an equivalence of categories. For all objects $C$ in the category $\mathtt{C}$, we denote by 
	\[
		\Sigma C:=\Sigma\otimes C
	\]
	its image by this functor. The reader who is not familiar with these categorical notions can think of $\mathtt{C}$ as the category $\Ch_k$ of $\mathbb{Z}$-graded chain complexes over a field $k$, equipped with the monoidal structure $\otimes$ given by the tensor product of complexes and symmetry given, for $a \in A$ and $b\in B$ homogeneous elements, by $\tau_{A,B}(a\otimes b)=(-1)^{|a||b|}b\otimes a$. In this case, the reader can take $\Sigma$ to be the chain complex $k$, concentrated in degree $r$, for $r$ a fixed integer, i.e. $(\Sigma A)_n=A_{n-r}= :A[-r]$.
	
	We denote by $\mathtt{As}(\C)$ the category of associative monoids in $\C$: objects in $\mathtt{As}(\C)$ are objects $A\in \C$ with an associative product $\mu: A\otimes A \rightarrow A$, and for two monoids $A$ and $B$, $\Hom_{\mathtt{As}(\C)}(A,B)$ is the set of monoid morphisms.

	We say the category $\mathtt{C}$ is \emph{closed}, if for all $C\in \Ob \mathtt{C}$, there exists a functor $\sHom(C,-)$ which is the right adjoint of $-\otimes~C$: 
	\begin{equation*}
		-\otimes~C :\C \rightleftarrows \C : \sHom(C,-).
	\end{equation*}
	For example, the category $(\Ch_k,\otimes_k)$ is closed (see \cite{Wei94}). 

\subsection{$A$-bimodule structures on $A\otimes A$}\label{subsect::bimodule_structure}
Fix $A$ and $B$, two associative monoids in $\C$ : the product $A\otimes B$ is also a monoid in $\C$. We define $A^{\circ}$, the opposite monoid of $A$, to be given by the same object but with the product $\mu_{A^\circ}:=\mu_A\circ \tau_{A,A}$. We have the usual notion of left or right modules over $A$: we denote by $A$-$\mathtt{Mod}_\C$ (resp. $\mathtt{Mod}_\C$-$A$) the category of left $A$-modules (resp. right $A$-modules) in the category $\C$ and we have the equivalence $	\mathtt{Mod}_\C\mbox{-}A \cong A^\circ\mbox{-}\mathtt{Mod}_\C$. We denote by $(A,B)$-$\mathtt{Bimod}_\C:=A$-$\mathtt{Mod}_\C$-$B$, the category of $(A,B)$-bimodules in $\C$, which is equivalent to $(A\otimes B^\circ)$-$\mathtt{Mod}_\C$. For $M$ an $(A,B)$-bimodule and $X$ and $Y$ two objects in $\C$, the product  $X\otimes N\otimes Y$ is also an $(A,B)$-bimodule by the symmetry of $\C$. Fix two $(A,B)$-bimodules $M$ and $N$: the product $M\otimes N$ has a structure of $(A,B)$-bimodule, called  the \emph{external} one, given by the left $A$-action on $M$ and the $B$-right action on $N$. $M\otimes N$ also has an \emph{internal} $(A,B)$-bimodule structure, given by by the left $A$-action on $N$ and the right $B$-action on $M$.
 
When $A=B$, we denote $A\mbox{-}\mathtt{Bimod}_\C:= (A,A)\mbox{-}\mathtt{Bimod}_\C$, the category of $(A,A)$-bimodules. We denote by $A^e$, the monoid $A^e:= A\otimes A^\circ$, so that the category $A\mbox{-}\mathtt{Bimod}_\C$ is equivalent to $A^e$-$\mathtt{Mod}_\C$. The symmetry of the category $\C$ gives us the isomorphism of monoids $\tau_{A,A}: A^e \cong (A^e)^{\circ}$.  The monoid structure 	of  $A\otimes A^\circ$ gives a canonical structure of $A^e$-bimodule on $\Sigma_1A\otimes \Sigma_2 A$ for all $\Sigma_1,\Sigma_2$ in $\C$, i.e. two $A$-bimodule structures (we implicitly use the isomorphism $(A^e)^\circ \cong A^e$):
		\begin{enumerate}
			\item \emph{the external structure} given by
			\begin{eqnarray*}
			&\mu_A^e :=(\Sigma_1 A\otimes \Sigma_2\mu): (\Sigma_1 A\otimes\Sigma_2 A)\otimes A \rightarrow \Sigma_1 A\otimes\Sigma_2 A, \\
			&_A^{~}\mu^e := (\Sigma_1\mu\otimes \Sigma_2 A)(\tau_{A,\Sigma_1}\otimes A\otimes\Sigma_2 A): A\otimes (\Sigma_1 A\otimes\Sigma_2 A) \rightarrow \Sigma A\otimes\Sigma A ;
			\end{eqnarray*}
			\item \emph{the internal structure} given by
			\begin{eqnarray*}
			&\mu_A^i := (\Sigma_1\mu\otimes\Sigma_2 A)(\Sigma_1 A\otimes \tau_{\Sigma_2 A,A}) : (\Sigma_1 A\otimes\Sigma_2 A)\otimes A \rightarrow \Sigma_1 A\otimes\Sigma_2 A, \\
			&_A^{~}\mu^i := (\Sigma_1 A\otimes\Sigma_2 \mu)(\tau_{A,\Sigma_1 A \Sigma_2} \otimes  A):A\otimes (\Sigma_1 A\otimes \Sigma_2 A) \rightarrow \Sigma_1 A\otimes\Sigma_2 A .
			\end{eqnarray*}
		\end{enumerate}
\begin{rem}\label{rem::equivalence_categorie}
	Let $\Sigma$ an invertible object in $\C$ with the isomorphism $\rho: \Sigma^-\otimes \Sigma \rightarrow \mathds{1}$ and $A$ be a monoid in $\C$. The functor $\Sigma\otimes -:\C \rightarrow \C$ is equivalent to an equivalence of categories $\Sigma\otimes -:A\mbox{-}\mathtt{Mod}_\C \rightarrow A\mbox{-}\mathtt{Mod}_\C$.
\end{rem}

\section{\texorpdfstring{$\Sigma$}{Sigma}-double Poisson algebras}
In this section, we extend constructions made by Van den Bergh in \cite{VdB08} to a general categorical framework. We consider $(\mathtt{C},\otimes)$ a ymmetric monoidal additive category and we fix $\Sigma$ an invertible object in $\C$ with the isomorphism $\rho : \Sigma^-\otimes \Sigma \rightarrow \mathds{1}$  and $(A,\mu)$ a monoid in $\C$. 

\subsection{$\Sigma$-double bracket}
We recall that a morphism $\phi : A \rightarrow M$ in $\C $ between an algebra $(A,\mu)$ and an $A$-bimodule $(M,\mu_A,_A\!\mu)$ is a \emph{derivation} if $\phi\mu= \mu_A (\phi\otimes A)+ _A\mu(A\otimes \phi)$. We denote by $\Der(A,M)$ (resp. $\Der(A)$) the group of derivations between $A$ and $M$ (resp. the group of derivations between $A$ and itself). 
\begin{rem}
	If the category $\C$ is closed, we can internalize the notion of derivation, then  the group $\Der(A,M)$ is in $\C$. 	For example, for a differential graded algebra $A$ and an $A$-bimodule $M$ in $\Ch_k$, the set $\Der(A,M)$ can be enriched with a chain complex structure.  
\end{rem}
\begin{definition}[$\Sigma$-shifted double bracket]\label{def::dcrochet_decale}
	Let $(A,\mu)$ be a monoid of $(\mathtt{C},\otimes)$. A \emph{$\Sigma$-shifted double bracket} or $\Sigma$-\emph{double bracket} on $A$ is a morphism
	\[
		f:=\DB{-}{-}: \Sigma A \otimes \Sigma A \longrightarrow \Sigma A\otimes A,
	\]
	represented by the directed coloured graph
	
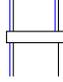
\[
	\begin{tikzpicture}[scale=0.3]
	\draw[shift={(-0.15,0)},thin,blue] (2,9.5) -- (2,6);
	\draw[thin] (2,9.5) -- (2,6);
	\draw[shift={(-0.15,0)},thin,blue] (4,9.5) -- (4,7.5);
	\draw[thin] (4,9.5) -- (4,6);
	\draw[fill=white](1.7,7.5) rectangle (4.3,8);
	\end{tikzpicture}
	\]
	where the direction is from top to bottom and where blue edges represent the suspension $\Sigma$. The $\Sigma$-double bracket $f$	
	\begin{itemize}
		\item is \emph{antisymmetric} if $\DB{-}{-}= - \Sigma\tau_{A,A} \DB{-}{-} \tau_{\Sigma A,\Sigma A}$, i.e. in terms of directed graphs:
		\[
		\begin{tikzpicture}[scale=0.3,baseline=(n.base)]
			\draw[shift={(-0.15,0)},thin,blue] (2,9.5) -- (2,6);
			\draw[thin] (2,9.5) -- (2,6);
			\draw[shift={(-0.15,0)},thin,blue] (4,9.5) -- (4,8);
			\draw[thin] (4,9.5) -- (4,6);
			\draw[fill=white] (1.7,7.5) rectangle (4.3,8);
			\node (n) at (3,7.75) {};
		\end{tikzpicture}
		\quad = - \quad
		\begin{tikzpicture}[scale=0.3,baseline=(n.base)]
			\draw (2,9.5) to[out=270,in=90] (4,8);
			\draw[shift={(-0.15,0)},thin,blue] (2,9.5) to[out=270,in=90] (4,8);
			\draw (4,9.5) to[out=270,in=90] (2,8);
			\draw[shift={(-0.15,0)},thin,blue] (4,9.5) to[out=270,in=90] (2,8);
			\draw[fill=white] (1.7,7.5) rectangle (4.3,8);
			\node (n) at (3,7.75) {};
			\draw (2,7.5) to[out=270,in=90] (4,6);
			\draw (4,7.5) to[out=270,in=90] (2,6);
			\draw[shift={(-0.15,0)},thin,blue] (2,7.5) -- (2,6);
		\end{tikzpicture}~~;
		\]
		\item satisfies \emph{the left derivation property} if the double bracket $f$ is a derivation in its first variable for the internal $A$-bimodule structure of $\Sigma A\otimes A$, i.e.  
		\[
		 f(\Sigma \mu\otimes \Sigma A)= ^{~}_A\mu^i(A\otimes f)(\tau_{\Sigma, A}\otimes A\otimes \Sigma A) +\mu^i_A (f\otimes A)(\Sigma A\otimes \tau_{A,\Sigma A}).
		\]
		This property can be described by the following sum of directed graphs:
		\[
			\begin{tikzpicture}[scale=0.3,baseline=(n.base)]
				\node (n) at (3,7) {};
				\draw[shift={(-0.1,0)},thin,blue] (0,10) to[out=270,in=100] (1,7.5); 
				\draw[thin] (0,10) to[out=270,in=100] (1,8);
				\draw[thin] (2,10) to[out=270,in=80] (1,8);
				\coordinate (Mu) at (1,8.1) ;
				\draw[fill=black] (Mu) circle (3pt);
				\draw[thin] (1,8) to[out=270,in=90] (1,7.5);
				\draw[thin] (4,10) to[out=270,in=90] (3,8) to[out=270,in=90] (3,7.5);
				\draw[shift={(-0.1,0)},blue,thin] (4,10) to[out=270,in=90] (3,8) to[out=270,in=90] (3,7.5);
				\draw (0.7,7) rectangle (3.3,7.5);
				\draw[thin] (1,7) -- (1,5);
				\draw[shift={(-0.1,0)},blue,thin] (1,7) -- (1,5);
				\draw[thin] (3,7) -- (3,5);
			\end{tikzpicture}
			\quad = \quad
			\begin{tikzpicture}[scale=0.3,baseline=(n.base)]
			\node (n) at (3,7) {};
				\draw (0,10) -- (0,9.4);
				\draw[thin,blue] (-0.1,10) -- (-0.1,9.5); 
				\draw[thin,blue] (-0.1,9.5) to[out=270,in=100] (1.9,8); 
				\draw (0,9.4) -- (0,7.5);
				\draw[thin] (2,10) -- (2,8);
				\draw[thin] (4,10) -- (4,8);
				\draw[shift={(-0.1,0)},blue,thin] (4,10) -- (4,8);
				\draw (1.7,7.5) rectangle (4.3,8);
				\draw[thin,blue] (1.9,7.5) to[out=270,in=90] (-0.1,6.5);
				\draw (0,7.5) to[out=270,in=90] (3,5.5);
				\draw[thin] (2,7.5) to[out=270,in=90] (1,5.5);
				\coordinate (Mu) at (3,5.6) ;
				\draw[fill=black] (Mu) circle (3pt);
				\draw[shift={(-0.1,0)},blue,thin] (0,6.5) to[out=270,in=90] (1,5);
				\draw[thin] (1,5.5) -- (1,5);
				\draw[thin] (4,7.5) to[out=270,in=90] (3,5.5);
				\draw[thin] (3,5.5) -- (3,5);
			\end{tikzpicture}
			\quad + \quad
			\begin{tikzpicture}[scale=0.3,baseline=(n.base)]
				\node (n) at (3,7) {};
				\draw[thin] (2,10) -- (2,9.5);
				\draw[thin] (4,10) -- (4,9.5);
				\draw[shift={(-0.1,0)},blue,thin] (4,10) -- (4,9.5);
				\draw[thin] (0,10) -- (0,8);
				\draw[shift={(-0.1,0)},blue,thin] (0,10) -- (0,8);
				\draw[thin] (4,9.5) to[out=270,in=90] (2,8);
				\draw[shift={(-0.1,0)},blue,thin] (4,9.5) to[out=270,in=90] (2,8);
				\draw (2,9.5) to[out=270,in=90] (4,8);
				\draw[thin] (4,8) -- (4,7.5);
				\draw (-0.3,7.5) rectangle (2.3,8);
				\draw[thin] (0,7.5) to[out=270,in=90] (1,5.5);
				\draw[shift={(-0.1,0)},blue,thin] (0,7.5) to[out=270,in=90] (1,5);
				\draw[thin] (2,7.5) to[out=270,in=90] (3,5);
				\draw[thin] (4,7.5) to[out=270,in=90] (1,5.5);
				\coordinate (Mu) at (1,5.6) ;
				\draw[fill=black] (Mu) circle (3pt);
				\draw[thin] (1,5.5) -- (1,5);
			\end{tikzpicture}
		\]
	\end{itemize}
	If $\Sigma=\mathds{1}$, a $\Sigma$-shifted double bracket is just called a double bracket.
\end{definition}
\begin{rem}
	Suppose that $(A,\mu,\iota)$ is a monoid with unit $\iota : \mathds{1}\rightarrow A$ in the category $\mathtt{C}$, with a  $\Sigma$-double bracket $f$ which satisfies the derivation property. 
	Then, as the following diagram commutes
	\[
	\xymatrix@C=1.5pc@R=0.7pc{
		\mathds{1}\otimes \mathds{1} \ar@{}[rd]|\commu \ar[r]^{\iota\otimes\iota} \ar[d]_\cong & A\otimes A \ar[d]^\mu \\
		\mathds{1} \ar[r]_\iota & A,
	}
	\]
	the morphism $ \mathds{1}\otimes A \overset{f(\iota\otimes1)}{\longrightarrow} A\otimes A$ is trivial.
\end{rem}
\begin{prop} \label{derivation}
	Let $(A,\mu)$ be a monoid of the category $\mathcal{C}$, with an antisymmetric $\Sigma$-double bracket $f$ which satisfies the left derivation property. Then, the $\Sigma$-double bracket $f$ also satisfies the \emph{right derivation property}, which can be described by the following sum of directed graphs :
	\[
	\begin{tikzpicture}[scale=0.3,baseline=(n.base)]
		\node (n) at (0.7,7) {};
		\draw[thin] (0,10) to[out=270,in=90] (1,7.5);
		\draw[shift={(-0.1,0)},blue,thin] (0,10) to[out=270,in=90] (1,7.5);
		\draw[thin] (2,10) to[out=270,in=100] (3,8);
		\draw[thin] (4,10) to[out=270,in=80] (3,8);
		\draw[shift={(-0.1,0)},blue,thin] (2,10) to[out=270,in=90] (3,7.5);
		\draw[fill=black] (3,8.1) circle (3pt);
		\draw[thin] (3,8) -- (3,7.5);
		\draw (0.7,7) rectangle (3.3,7.5);		
		\draw[thin] (1,7) -- (1,5);
		\draw[shift={(-0.1,0)},blue,thin] (1,7) -- (1,5);
		\draw[thin] (3,7) -- (3,5);
	\end{tikzpicture}
	\quad = \quad
	\begin{tikzpicture}[scale=0.3,baseline=(n.base)]
		\node (n) at (0.7,7) {};
		\draw[thin] (0,10) -- (0,9.5);
		\draw[shift={(-0.1,0)},blue,thin] (0,10) -- (0,9.5);
		\draw (2,10) -- (2,9.4);
		\draw[shift={(-0.1,0)},blue,thin] (2,10) -- (2,9.5);
		\draw[thin] (4,10) -- (4,8);
		\draw[thin] (0,9.5) to[out=270,in=90] (2,8);
		\draw[shift={(-0.1,0)},blue,thin] (0,9.5) to[out=270,in=90] (2,8);
		\draw[shift={(-0.1,0)},blue,thin] (2,9.5) to[out=270,in=90] (4,8);
		\draw (2,9.4) to[out=270,in=90] (0,8);
		\draw[thin] (0,8) -- (0,7.5);
		\draw (1.7,7.5) rectangle (4.3,8);
		\draw[thin,blue] (1.9,7.5) to[out=270,in=90] (-0.1,6.5);
		\draw (0,7.5) to[out=270,in=90] (1,5.5);
		\draw[thin] (2,7.5) to[out=270,in=90] (1,5.5);
		\draw[fill=black] (1,5.6) circle (3pt);
		\draw[shift={(-0.1,0)},blue,thin] (0,6.5) to[out=270,in=90] (1,5);
		\draw[thin] (1,5.5) -- (1,5);
		\draw[thin] (4,7.5) to[out=270,in=90] (3,5.5);
		\draw[thin] (3,5.5) -- (3,5);			
	\end{tikzpicture}
	\quad + \quad
	\begin{tikzpicture}[scale=0.3,baseline=(n.base)]
		\node (n) at (0.7,7) {};
		\draw[thin] (0,10) -- (0,8);
		\draw[shift={(-0.1,0)},blue,thin] (0,10) -- (0,8);
		\draw[thin] (2,10) -- (2,8);
		\draw[shift={(-0.1,0)},blue,thin] (2,10) -- (2,8);
		\draw[thin] (4,10) -- (4,7.5);
		\draw (-0.3,7.5) rectangle (2.3,8);
		\draw[thin] (0,7.5) to[out=270,in=90] (1,5);
		\draw[shift={(-0.1,0)},blue,thin] (0,7.5) to[out=270,in=90] (1,5);
		\draw[thin] (2,7.5) to[out=270,in=90] (3,5.5)-- (3,5);
		\draw[thin] (4,7.5) to[out=270,in=90] (3,5.5);
		\coordinate (Mu) at (3,5.6) ;
		\draw[fill=black] (Mu) circle (3pt);
	\end{tikzpicture}~.
	\]
\end{prop}

\begin{definition}[$\Sigma$-double Lie algebra]\label{def::algebre_doubleLie}
	Let $L$ be an object of the category $\mathtt{C}$, with an antisymmetric $\Sigma$-double bracket 
	$
		f:=\DB{-}{-} : \Sigma L \otimes \Sigma L \longrightarrow \Sigma L\otimes L.
	$
	We call \emph{double-jacobiator} associated to $f$, the application
	\[
	DJ_f := \TB{-}{-}{-} : \Sigma L \otimes \Sigma L \otimes \Sigma L \rightarrow \Sigma L\otimes L\otimes L
	\]
	defined by 
	\[
		\DB{-}{-}_l + \Sigma\tau_{A,A\otimes A}\DB{-}{-}_l\tau_{\Sigma A\otimes \Sigma A, \Sigma A} +\Sigma\tau_{A\otimes A,}\DB{-}{-}_l\tau_{\Sigma A,\Sigma A\otimes \Sigma A}
	\]
	where $\DB{-}{-}_l= (f\otimes A)(\Sigma A\otimes f)$; we can describe the double jacobiator diagrammatically by the following sum of directed graphs:
	\[
	\begin{tikzpicture}[scale=0.3,baseline=(n.base)]
		\node (n) at (0.7,7) {};
		\draw (0,10) to[out=270,in=90] (0,8);
		\draw[shift={(-0.1,0)},thin,blue] (0,10) to[out=270,in=90] (0,8);
		\draw (2,10) to[out=270,in=90] (2,8);
		\draw[shift={(-0.1,0)},thin,blue] (2,10) to[out=270,in=90] (2,8);
		\draw (4,10) to[out=270,in=90] (4,8);
		\draw[shift={(-0.1,0)},thin,blue] (4,10) to[out=270,in=90] (4,8);
		\draw(1.7,7.5) rectangle (4.3,8);
		\draw[thin] (0,8) -- (0,7);
		\draw[shift={(-0.1,0)},thin,blue] (0,8) -- (0,7);
		\draw[thin] (2,7.5) -- (2,7);
		\draw[shift={(-0.1,0)},thin,blue] (2,7.5) -- (2,7);
		\draw[thin] (4,7.5) -- (4,6.5);
		\draw(-0.3,6.5) rectangle (2.3,7);
		\draw (0,6.5) to[out=270,in=90] (0,4.5);
		\draw (2,6.5) to[out=270,in=90] (2,4.5);
		\draw (4,6.5) to[out=270,in=90] (4,4.5);
		\draw[shift={(-0.1,0)},thin,blue] (0,6.5) -- (0,4.5);
	\end{tikzpicture}
	\quad + \quad
	\begin{tikzpicture}[scale=0.3,baseline=(n.base)]
		\node (n) at (0.7,7) {};
		\draw (0,10) to[out=270,in=90] (4,8);
		\draw[shift={(-0.1,0)},thin,blue] (0,10) to[out=270,in=90] (4,8);
		\draw (2,10) to[out=270,in=90] (0,8);
		\draw[shift={(-0.1,0)},thin,blue] (2,10) to[out=270,in=90] (0,8);
		\draw (4,10) to[out=270,in=90] (2,8);
		\draw[shift={(-0.1,0)},thin,blue] (4,10) to[out=270,in=90] (2,8);
		\draw(1.7,7.5) rectangle (4.3,8);
		\draw[thin] (0,8) -- (0,7);
		\draw[shift={(-0.1,0)},thin,blue] (0,8) -- (0,7);
		\draw[thin] (2,7.5) -- (2,7);
		\draw[shift={(-0.1,0)},thin,blue] (2,7.5) -- (2,7);
		\draw[thin] (4,7.5) -- (4,6.5);
		\draw(-0.3,6.5) rectangle (2.3,7);
		\draw (0,6.5) to[out=270,in=90] (2,4.5);
		\draw (2,6.5) to[out=270,in=90] (4,4.5);
		\draw (4,6.5) to[out=270,in=90] (0,4.5);
		\draw[shift={(-0.1,0)},thin,blue] (0,6.5) -- (0,4.5);
	\end{tikzpicture}
	\quad + \quad
	\begin{tikzpicture}[scale=0.3,baseline=(n.base)]
		\node (n) at (0.7,7) {};
		\draw (0,10) to[out=270,in=90] (2,8);
		\draw[shift={(-0.1,0)},thin,blue] (0,10) to[out=270,in=90] (2,8);
		\draw (2,10) to[out=270,in=90] (4,8);
		\draw[shift={(-0.1,0)},thin,blue] (2,10) to[out=270,in=90] (4,8);
		\draw (4,10) to[out=270,in=90] (0,8);
		\draw[shift={(-0.1,0)},thin,blue] (4,10) to[out=270,in=90] (0,8);
		\draw(1.7,7.5) rectangle (4.3,8);
		\draw[thin] (0,8) -- (0,7);
		\draw[shift={(-0.1,0)},thin,blue] (0,8) -- (0,7);
		\draw[thin] (2,7.5) -- (2,7);
		\draw[shift={(-0.1,0)},thin,blue] (2,7.5) -- (2,7);
		\draw[thin] (4,7.5) -- (4,6.5);
		\draw(-0.3,6.5) rectangle (2.3,7);
		\draw (0,6.5) to[out=270,in=90] (4,4.5);
		\draw (2,6.5) to[out=270,in=90] (0,4.5);
		\draw (4,6.5) to[out=270,in=90] (2,4.5);
		\draw[shift={(-0.1,0)},thin,blue] (0,6.5) -- (0,4.5);
	\end{tikzpicture}~.
	\]
	We say $f$ is a $\Sigma$-double Lie bracket if $DJ_f=0$. In this case, we say $L$ is a $\Sigma$-\emph{double Lie} algebra.
	For $(L,\DB{-}{-}_L)$ and $(H,\DB{-}{-}_H)$, two $\Sigma$-double Lie algebras, a morphism $\phi : L \rightarrow H$ of $\mathtt{C}$ is a  $\Sigma$-double Lie algebra morphism if the following diagram commutes:
	\[
	\xymatrix@C=3pc{
		\Sigma L \otimes \Sigma L \ar[r]^{\Sigma \phi\otimes \Sigma \phi} \ar[d]_{\DB{-}{-}_L} \ar@{}[rd]|\commu & 
		\Sigma H \otimes \Sigma H \ar[d]^{\DB{-}{-}_H} \\
		\Sigma L\otimes L \ar[r]_{\Sigma\phi\otimes\phi} & \Sigma H\otimes H~;
	}
	\]
	we denote by $\Sigma$-$\catDLie_\C$ the category of $\Sigma$-double Lie algebras. 
\end{definition}

\begin{rem}\label{rem::dJacobi_stable_Z3Z}
	The double jacobiator is stable under the diagonal action of $\Z/3\Z$, i.e. $DJ_f = \Sigma\tau_{A,A\otimes A}DJ_f\tau_{\Sigma A\otimes \Sigma A, \Sigma A}$.
\end{rem}

\begin{definition}[The category $\Sigma$-$\DPoiss_\mathtt{C}$]\label{def::doubleJacobi}
	Let $(A,\mu)$ be a monoid of $\mathtt{C}$ with a  $\Sigma$-double Lie bracket $f:=\DB{-}{-} : \Sigma A \otimes \Sigma A \longrightarrow \Sigma A\otimes A$,	which also satisfies the left derivation property: such a $\Sigma$-double bracket is called a \emph{$\Sigma$-double Poisson bracket} and $A$ is a \emph{$\Sigma$-double Poisson algebra} in the category $\mathtt{C}$.
	Let $(A,f_A)$ and $(B,f_B)$ be two $\Sigma$-double Poisson algebras  and $\phi : A \rightarrow B$ a morphism in $\mathtt{C}$. We say $\phi$ is a \emph{$\Sigma$-double Poisson algebra morphism} if $\phi$ is a monoid morphism and a $\Sigma$-double Lie morphism. We denote by $\Sigma$-$\DPoiss_\mathtt{C}$ the category of $\Sigma$-double Poisson algebras in $\mathtt{C}$, so that there are forgetful functors
	\[
	\Sigma\mbox{-}\DPoiss_\mathtt{C} \longrightarrow \Sigma\mbox{-}\mathtt{DLie}_\C \quad \mbox{and} \quad \Sigma\mbox{-}\DPoiss_\mathtt{C} \longrightarrow \mathtt{As}(\C).
	\]
\end{definition}

\begin{definition}[Left $\Sigma$-Leibniz algebra]
	Let $L$ be a object in $\C$, $\Sigma$ an invertible object in $\C$ and $f:\Sigma L\otimes \Sigma L \rightarrow \Sigma L$. We say $(L,f)$ is a \emph{left $\Sigma$-Leibniz algebra} if $f$ satisfies the Leibniz identity:
	\[
		f(\Sigma A\otimes f)= f(f\otimes \Sigma A) + f(\Sigma A\otimes f)(\tau_{\Sigma A,\Sigma A}\otimes \Sigma A).
	\]
\end{definition}

\begin{prop}[cf. \cite{VdB08}]\label{double_Poisson_Leibniz}
	Let $(A,\mu)$ be a monoid in  $\mathtt{C}$ equipped with a $\Sigma$-double Poisson bracket $f$. Then $(\Sigma A,\Sigma\mu f)$ is a $\Sigma$-left Leibniz algebra in $\mathtt{C}$. 
\end{prop}

\subsection{Double Poisson structure on a free monoid}\label{sect::DP_free_monoid}
Fix $(A,\mu)$ a monoid in $\mathtt{C}$ and $M$ an $A$-bimodule. We consider the free $A$-algebra on $M$:
\[
T_A(M):= A\oplus \bigoplus_{n\in\N^*} M^{\otimes_A n} 
\]
satisfying the following universal property: for an $A$-algebra $B$ with an $A$-bimodule morphism $M\rightarrow B$, we have the following canonical extension
\[
\xymatrix{
	M \ar[r] \ar@{^{(}->}[d] & B \\
	T_A(M) \ar@{.>}@/_1pc/[ru]_{\exists !\phi} 
}
\]
with $\phi$ an $A$-algebra morphism. We have the injection $A\oplus M \hookrightarrow T_A M$. Then, we have the following result:
\begin{lem}
	Let $A$ be a monoid and $M$ an $A$-bimodule. An antisymmetric $\Sigma$-double bracket on $T_A(M)$ that satisfies the left derivation property is determined by its restrictions to $\Sigma A\otimes \Sigma A$, $\Sigma M\otimes\Sigma  A$ and $\Sigma M\otimes\Sigma  M \subset \Sigma T_A(M)\otimes  \Sigma T_A(M)$.
\end{lem}

Let $A$ and $M$ fixed in $\C$, with $A$ a monoid $M$ an $A$-bimodule and let $\DB{-}{-}$ be a $\Sigma$-double bracket on $T_A(M)$. We will define several classes of Poisson double brackets on $T_A(M)$: we name these double bracket as in \cite{ORS13}.

\subsubsection{Constant double bracket}

\begin{definition}[Constant double Poisson bracket]
	We say $\DB{-}{-}$ is a \emph{constant} Poisson double bracket if its restrictions to $\Sigma A\otimes \Sigma A$ and $\Sigma M\otimes \Sigma A$ are null and if its restriction to $\Sigma M\otimes\Sigma M$ takes values in $\Sigma A\otimes A$, i.e. $\DB{-}{-}$ is completely defined by the antisymmetric $\Sigma$-double bracket
	\[
	\DB{-}{-} : \Sigma M\otimes \Sigma M \longrightarrow \Sigma A\otimes A .
	\]
\end{definition}
\begin{exmp}
	If $A=\mathds{1}$, a constant $\Sigma$-double Poisson bracket on  $T_{\mathds{1}}(M)$, corresponds to an antisymmetric bilinear form on $\Sigma M$.
\end{exmp}

\subsubsection{Linear double bracket}

\begin{definition}[Linear double Poisson bracket]
		We say $\DB{-}{-}$ is a \emph{linear} Poisson double bracket if its restriction to $\Sigma A\otimes \Sigma A$ is null and its restrictions to $\Sigma M\otimes \Sigma A$ and $\Sigma M\otimes\Sigma M$ take values respectively in  $\Sigma A\otimes A$  and $\Sigma (M\otimes A\oplus A\otimes M)$, that is $\DB{-}{-}$ is determined by 
	\begin{align*}
	\DB{-}{-} : & \Sigma M\otimes \Sigma A \longrightarrow \Sigma A\otimes A \quad \mbox{and} \\
	\DB{-}{-} : & \Sigma M\otimes \Sigma M \longrightarrow \Sigma (M\otimes A\oplus A\otimes M). 
	\end{align*}
	We denote by $\Sigma\mbox{-}\mathtt{DPFree}_A^{lin}$ the category where objects are free $A$-algebras with a linear $\Sigma$-double Poisson bracket, the morphisms are  $\Sigma$-double Poisson algebra morphisms induced by an $A$-bimodule morphism.
\end{definition}

\begin{exmp}\label{exmp::DLR_associatif}
	(cf. \cite[Sect. 2]{ORS13}) We consider $\DB{-}{-}$, a linear $\mathds{1}$-double Poisson bracket on $T_{\mathds{1}} M$, which is determined by morphisms $f: M\otimes M \longrightarrow M\otimes \mathds{1}$ and $g: M\otimes \mathds{1}  \longrightarrow \mathds{1} \otimes \mathds{1}$. By the derivation property of $\DB{-}{-}$, the morphism $g$ is null and the double Jacobi identity gives us the identity
	\[
	pr_{M\otimes \mathds{1}\otimes \mathds{1}}\TB{-}{-}{-}|_{M^{\otimes 3}} =0;
	\]
	so we have $(f\otimes \mathds{1})(M\otimes f) - (\tau_{\mathds{1},M}\otimes \mathds{1})(\mathds{1}\otimes f)(f\otimes M) = 0$, which is equivalent to
	\[
	f(M\otimes f) = f(f\otimes M).
	\]
	This corresponds to an associative monoid structure on $M$ (without unit). Futhermore, we have the equivalence of categories 
	\[
	\mathds{1}\mbox{-}\mathtt{DPFree}_\mathds{1}^{lin} \cong \mathtt{As}(\C).
	\]
\end{exmp}

\subsubsection{Quadratic double Poisson bracket}

\begin{definition}[Quadratic double Poisson bracket]\label{def::dPoisson_quadratique} 
	We say that $\DB{-}{-}$ is a \emph{quadratic} double Poisson bracket if its restrictions to $\Sigma A\otimes \Sigma A$ and $\Sigma M\otimes \Sigma A$ are trivial and its restriction to $\Sigma M\otimes\Sigma M$ takes values in $\Sigma M\otimes M$, i.e. $\DB{-}{-}$ is determined by the antisymmetric $\Sigma$-double bracket
	\[
	\DB{-}{-} :  \Sigma M\otimes \Sigma M \longrightarrow \Sigma M\otimes M. 
	\]		
	We denote by $\Sigma\mbox{-}\mathtt{DPFree}_A^{quad}$ the subcategory of free $A$-algebras with a quadratic $\Sigma$-double Poisson bracket, where morphisms are induced by $A$-bimodules morphisms.
\end{definition}
\begin{prop}
	The free associative functor $T_{\mathds{1}}(-)$ induces an equivalence of categories
	\[
	T_\mathds{1}(-) : \Sigma\mbox{-}\mathtt{DLie}_\C \overset{\cong}{\longrightarrow} \Sigma\mbox{-}\mathtt{DPFree}_{\mathds{1}}^{quad}.
	\]
\end{prop}
\begin{proof}
	We extend the $\Sigma$-double Lie bracket by derivation.
\end{proof}
\begin{exmps}
	\begin{enumerate}
		\item In \cite[Sect. 2.1]{ORS13}, Odesskii \emph{et al.} give a complete classification of quadratic double Poisson brackets on $\mathbb{C}\langle x,y\rangle$ with $|x|=|y|=0$.
		\item In \cite{Sok13}, Sokolov gives a complete classification of quadratic double Poisson brackets on $\mathbb{C}\langle x,y,z\rangle$ with $|x|=|y|=|z|=0$.
	\end{enumerate} 
\end{exmps}

\section{Double Lie-Rinehart algebras}
In this section, we extend the notion of a double Lie-Rinehart algebra, first defined by Van den Bergh in \cite{VdB08-2}, which is a noncommutative version of Lie-Rinehart algebras.

\subsection{Recollections on Lie-Rinehart algebras}
For a more complete exposition, the reader is refered to \cite{Kap07}, \cite[Sect. 13.3.8]{LV12} et \cite[Sect. 5.1.2]{Alv15}.
\begin{definition}[Lie-Rinehart algebra]
	\label{def::Lie_Rinehart_algebra}
	Let $(\C,\otimes,\tau)$ be an additive symmetric monoidal category and $A$ a commutative monoid in $\C$. A Lie algebra $(\g,[-,-])$ is a \emph{Lie-Rinehart algebra} over $A$ if $\g$ is an $A$-module for $_A\mu :A\otimes \g \rightarrow \g $, if $A$ is a $\g $-module for $\rho: \g\otimes A \rightarrow A$ which is called \emph{the anchor} and these module structures are compatible, i.e. satisfy the following properties.
	\begin{enumerate}
		\item \label{def::LR_1} The Lie algebra $\g $ acts by derivations on $A$, i.e. satisfies 
			\[
				\rho\circ \mu=\mu(A\otimes \rho)(\tau_{A,\g}\otimes A) + \mu(\rho\otimes A).
			\]
		\item \label{def::LR_2} We have the compatibility in $\Hom(\g^{\otimes 2}\otimes A,A)$:
			\[
				\rho\big([-,-]\otimes A\big)=\rho(\g\otimes \rho) - \rho(\g\otimes\rho)(\tau_{\g,\g}\otimes A).
			\]
		\item We have the \emph{Leibniz relation}:
		\[
		[-,-](\g\otimes _A\!\mu)=[-,-](\rho\otimes \g) +_A\!\!\mu(A \otimes[-,-])(\tau_{A,\g}\otimes \g).
		\] 
	\end{enumerate}
	Let $(M,\rho_M,[-,-]_M)$ and $(N,\rho_N,[-,-]_N)$ be two Lie-Rinehart algebras over $A$. An $A$-module morphism $\phi : M \rightarrow N$ is a \emph{morphism of Lie-Rinehart algebras} if $\phi$ commutes with the Lie bracket.
\end{definition}
\begin{rem}
	In the case where $\C$ is closed, the condition \eqref{def::LR_1} is equivalent to the anchor correspond to a morphism $\rho^*:\g \rightarrow \Der(A)$. Then, the condition \eqref{def::LR_2} is equivalent to $\rho^*$ is a Lie algebra morphism.
\end{rem}

For examples, the reader is refered to \cite[Ex. 1.3.3]{Kap07} or \cite[Chap. 5]{Alv15}.

\begin{prop}[{cf. \cite[Prop. 13.3.8]{LV12}}]\label{prop::LieRin_Pois}
	Any  Lie-Rinehart algebra $(A,L)$ gives rise to a Poisson algebra $P= A\oplus L$, where $A\oplus L$ is the square-zero extension as algebra, such that the two operations $\mu$ and $\{-,-\}$ take values as follows:
	\[
	\begin{array}{rl}
	A\otimes A \overset{\mu}{\longrightarrow} A, &  A\otimes A \overset{\{-,-\}}{\longrightarrow} 0, \\
	A\otimes L \overset{\mu}{\longrightarrow} L, &  L\otimes A \overset{\{-,-\}}{\longrightarrow} A, \\
	L\otimes L \overset{\mu}{\longrightarrow} 0, &  L\otimes L \overset{\{-,-\}}{\longrightarrow} L. 
	\end{array}
	\]	
	Conversely, any Poisson algebra $P$, whose underlying vector space can be split as $P= A\oplus L$ and such that the two operations take values as indicated above, defines a Lie-Rinehart algebra. The two constructions are inverse to each other.
\end{prop}
\begin{rem} (cf. \cite[Prop. 3.6.2]{VdL04})
This result is an operadic one. In fact, a Lie-Rinehart algebra is an algebra over the two-coloured operad $\mathcal{L}ie\mathcal{R}in$.
\end{rem}

\subsection{Double Lie-Rinehart algebras}\label{subsect::double_Lie_Rinehart}
Now, we extend the definition of the noncommutative version of Lie-Rinehart algebras: the double Lie-Rinehart algebras (called double Lie algebroids by Van den Bergh in \cite{VdB08-2}).
\begin{nota}
	Let $A$ a monoid in an additive symmetric monoidal category $(\C,\otimes,\tau)$, $M$ and $N$ two $A$-bimodules and $\phi:M\rightarrow N$ a morphism of $A$-bimodules. We canonically extend the morphism $\phi$ to an $A$-bimodule morphism $\tilde{\phi} : A\otimes M \oplus M\otimes A \longrightarrow A\otimes N \oplus N\otimes A$ where the structures of $A$-bimodule are induced by those of $M$ and $N$ and such that the restrictions to $A\otimes M$ and $M\otimes A$ are given by
	\[
		\tilde{\phi}|_{A\otimes M}= A\otimes \phi \quad \mbox{ and } \quad \tilde{\phi}|_{M\otimes A}= \phi\otimes A.
	\]
	Hereafter, we don't distinguish between $\phi$ and $\tilde{\phi}$.
\end{nota}
\begin{definition}[$\Sigma$-Double Lie-Rinehart algebra]\label{def::dLieRinehart}
	Let $(A,\mu)$ be a monoid in an additive symmetric monoidal category $(\C,\otimes,\tau)$ and $(M,_A\!\mu, \mu_A)$ an $A$-bimodule. We say that $M$ is a \emph{$\Sigma$-double Lie-Rinehart algebra over $A$} if we have
	\begin{enumerate}
		\item \label{def::anchor} an $A$-bimodule morphism called \emph{the anchor}
		\[
			\rho : \Sigma M\otimes \Sigma A \rightarrow\Sigma A\otimes A
		\]
		with, for the left term, the $A$-bimodule structure induced by that of $M$ and, for the right term, the internal structure, which is a derivation in the second input for the external $A$-bimodule structure on the right term;
		\item a morphism  
		\[
			\DB{-}{-}^M: \Sigma M\otimes \Sigma M \longrightarrow \Sigma (M \otimes A \oplus A\otimes M) 
		\]
		with components $\DB{-}{-}^M_l:= pr_{\Sigma M \otimes A}\circ \DB{-}{-}^M$ and $\DB{-}{-}^M_r:= pr_{\Sigma A \otimes M}\circ \DB{-}{-}^M$; 
		which satisfy the following conditions 
		\begin{enumerate}
		\item \label{def::DLR_antisym} \emph{antisymmetry} 
		\[
			\DB{-}{-}^M\tau_{\Sigma M, \Sigma M}= - \Sigma(\tau_{M,A}, \tau_{A,M}) \DB{-}{-}^M;
		\]
		\item \label{def::DRL_compatible1} the first compatibility with the anchor, called \emph{derivation property}: we have the following commutative diagram 
			\[
				\xymatrix{
					\Sigma M \otimes \Sigma A\otimes M \ar[r]^{\Sigma M \otimes \Sigma _A\!\mu }  \ar@/_1pc/[rd]_(0.4){\phi^l}
						& \Sigma M \otimes \Sigma M \ar[d]^{\DB{-}{-}^M} 
						& \Sigma M \otimes \Sigma M\otimes A \ar[l]_{\Sigma M \otimes \Sigma \mu_A } \ar@/^1pc/[ld]^(0.4){\phi^r} \\
					  & \Sigma( M\otimes A \oplus A\otimes M), & 
				}
			\]
			where
			\begin{align*}
				\phi^l:= & (\Sigma A\otimes _A\!\mu)(\rho\otimes M) \\
				& \qquad +( _A\mu^{\Sigma M}\otimes A, _A\mu^{\Sigma A}\otimes M)(A\otimes {\DB{-}{-}^M})(\tau_{\Sigma M \Sigma, A}\otimes M)  \\
				\phi^r:= & (\Sigma M\otimes \mu , \Sigma A \otimes \mu_A)({\DB{-}{-}^M}\otimes A)\\
				&\qquad + (\Sigma \mu_A\otimes A)(\tau_{M,\Sigma}\otimes A\otimes A)(M\otimes \rho)(\tau_{\Sigma M\Sigma,M}\otimes A);
			\end{align*}
			\item \label{def::DRL_compatible2} the second compatibility with the anchor: we have the following relation in $\Hom(\Sigma A\otimes (\Sigma M)^{\otimes 2}, \Sigma A^{\otimes 3})$
				\begin{eqnarray*}
					&&(\rho_\tau\otimes A)(\Sigma A\otimes \DB{-}{-}^M_l) \\
					&&\qquad +\quad \tau_{\Sigma A\otimes \Sigma A, \Sigma A}(\rho\otimes A)(\Sigma M\otimes \rho)\tau_{\Sigma A,\Sigma M \otimes \Sigma M }  \\
					&&\qquad + \quad\tau_{\Sigma A,\Sigma A\otimes \Sigma A}(\rho\otimes A)(\Sigma M\otimes \rho_\tau)\tau_{\Sigma A \otimes \Sigma M , \Sigma M } = 0
				\end{eqnarray*}
				where $\rho_\tau:=-(\Sigma\tau_{A,A})\rho\tau_{\Sigma A,\Sigma M}:\Sigma A\otimes \Sigma M \rightarrow \Sigma A\otimes A $. 
		 	\item \label{def::DRL_DJacobi} and the \emph{double-Jacobi} identity, which is the following relation in \\ $\Hom((\Sigma M)^{\otimes 3}, \Sigma M\otimes A^{\otimes 2})$:
		 	\begin{eqnarray*}
		 		&&(\DB{-}{-}_l^M\otimes A)(\Sigma M\otimes \DB{-}{-}_l^M) \\
		 		&&\qquad +\quad\Sigma\tau_{ A, M\otimes  A}(\DB{-}{-}^M_r\otimes A)(\Sigma M\otimes \DB{-}{-}^M_l)\tau_{\Sigma M \otimes \Sigma M , \Sigma M }  \\
		 		&&\qquad + \quad\Sigma\tau_{ A\otimes  A, M}(\rho\otimes A)(\Sigma M\otimes\DB{-}{-}^M_r)\tau_{\Sigma M,\Sigma M \otimes \Sigma M } = 0.
		 	\end{eqnarray*}
		\end{enumerate}
	\end{enumerate}
	Let $(M,\DB{-}{-}^M)$ and $(N,\DB{-}{-}^N)$ be two $\Sigma$-double Lie-Rinehart algebras over $A$. A \emph{$\Sigma$-double Lie-Rinehart algebra morphism} $\phi$ is an $A$-bimodule morphism which satisfies the property: 
	\[
		\phi\big( \DB{-}{-}^M\big) = \DB{\phi(-)}{\phi(-)}^N.
	\]
	We denote by $\Sigma\mbox{-}\mathtt{DLR}_A$ the category of $\Sigma$-double Lie-Rinehart algebras over $A$. 
\end{definition}
\begin{rem} 
	In the case where $\C$ is a \emph{closed} symmetric monoidal category (for example, $\C=\Ch_k$), then, by adjunction, the anchor of a $\Sigma$-double Lie-Rinehart algebra $A$ is equivalent to the $A$-bimodule morphism $\rho^*: \Sigma M \rightarrow \Der(\Sigma A, \Sigma A\otimes A)$. This morphism is the analogue of the anchor of a Lie-Rinehart algebra.
\end{rem}
\begin{rem}\label{rem::DLieRin_closed_cat} In the case $\C=\Ch_k$, when $A$ is a finitely generated differential graded associative algebra, the condition \eqref{def::DRL_compatible2} can be expressed using the Schouten-Nijenhuis double bracket (introduced in \ref{thm::dcrochet_schouten} below), as:
\[
	\rho^*\big( \DB{-}{-}^M\big) = \DB{\rho^*}{\rho^*}^{SN}.
\]
\end{rem}
\begin{exmps} 
	\begin{enumerate}
		\item For all finitely-generated differential graded associative algebra $A$, the $A$-bimodule $\DDer(A)$ of biderivations (cf. definition \ref{def::biderivation}) with the Schouten-Nijehuis double Poisson bracket and where the identity plays the role of the anchor, is a $0$-double Lie-Rinehart algebra over $A$.
		\item In  \cite[Sect. 5.5]{Alv15}, the noncommutative version of the Atiyah algebra is defined. Let $A$ be an associative $k$-algebra and $M$ a finitely-presented $A$-bimodule. We denote by $\EEnd(M)$, the $A$-bimodule $\Hom_{k}(M,M\otimes A \oplus A\otimes M)$, and, for $\phi$ in $\EEnd(M)$, we denote by $\phi_l:= pr_{M\otimes A}\circ \phi$ and $\phi_r:= pr_{A\otimes M}\circ \phi$, the compositions with the projections. The \emph{Atiyah double algebra on $M$}, denoted by  $\mathbb{A}\mathrm{t}(M)$, is the set of pairs $(d,\phi)$ with  $d\in \DDer(A)$ and $\phi\in \EEnd(M)$ with compatibilities analogous to the commutative case.
		\emph{The Atiyah double bracket} is define as follow: for  $(d^1,\phi^1)$ and $(d^2,\phi^2)$ in $\mathbb{A}\mathrm{t}(M)$ 
		\[
			\DB{(d^1,\phi^1)}{(d^2,\phi^2)}^{\mathbb{A}\mathrm{t}}:= \big( \DB{d^1}{d^2}^{SN}, \DB{(d^1,\phi^1)}{(d^2,\phi^2)}^{\mathbb{E}} \big),
		\]
		where $\DB{-}{-}^{SN}$ is the Schouten-Nijenhuis double Poisson bracket  (cf. section \ref{sect::Schouten}) and $\DB{-}{-}^{\mathbb{E}}$ is described in \cite{Alv15}. By \cite[Prop. 5.5.3]{Alv15}, $\mathbb{A}\mathrm{t}(M)$ equipped with the Atiyah double bracket and the anchor morphism given by 
		\[
			\begin{array}{rccc}
				\rho :& \mathbb{A}\mathrm{t}(M)  & \longrightarrow & \DDer(A) \\
				& (d,\phi) & \longmapsto & d
			\end{array}
		\]
		is a $0$-double Lie-Rinehart algebra over $A$ and $\rho$ is a morphism of double Lie Rinehart algebra.
	\end{enumerate}
\end{exmps}

\begin{rem}
	A double Lie-Rinehart algebra is an algebra over a (coloured) properad. Properads encode algebraic structures which have several inputs and outputs: they generalize the notion of operads which encode algebraic structures with several inputs and one output (see \cite{LV12}). Formally, a properad is a $\mathfrak{S}$-bimodule, i.e. a family of $\mathfrak{S}_n\times\mathfrak{S}_m^{op}$-module with $m$ and $n$ in $\N$, which is a monoid for the connected product $\boxtimes_c$, defined by Vallette in \cite{Val03,Val07}. This product is controlled by connected graphs: for two $\mathfrak{S}$-bimodules $P$ and $Q$, elements in $P\boxtimes_c Q$ can be described as sum of graphs as the following:
	
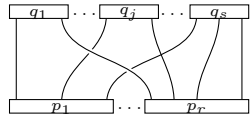
\[
	\begin{tikzpicture}[scale=0.3]
		\draw (0,2) to[out=270,in=90] (0,-2);
		\draw (4,2) to[out=270,in=90] (2,-2);
		\draw (6,2) to[out=270,in=90] (7,-2);
		\draw (8,2) to[out=270,in=90] (4,-2);
		\draw (9,2) to[out=270,in=90] (8,-2);
		\draw (10,2) to[out=270,in=90] (10,-2);
		\draw[draw=white,double=black,double distance=\pgflinewidth,thick] (2,2) to[out=270,in=90] (6,-2);
		\draw[fill=white] (-0.3,1.8) rectangle (2.3,2.4);
		\draw[fill=white] (3.7,1.8) rectangle (6.3,2.4); 
		\draw[fill=white] (7.7,1.8) rectangle (10.3,2.4); 
		\draw (3.1,2) node {\tiny{$\cdots$}};
		\draw (7,2) node {\tiny{$\cdots$}};
		\draw (1,2) node {\tiny{$q_1$}};
		\draw (5,2) node {\tiny{$q_j$}};
		\draw (9,2) node {\tiny{$q_s$}};
		\draw[fill=white] (-0.3,-2.4) rectangle (4.3,-1.8);
		\draw[fill=white] (5.7,-2.4) rectangle (10.3,-1.8);
		\draw (5.1,-2.2) node {\tiny{$\cdots$}};
		\draw (2,-2.2) node {\tiny{$p_1$}};
		\draw (8,-2.2) node {\tiny{$p_r$}};
	\end{tikzpicture}
	\]
	where $p_1,\ldots,p_r$ are elements in $P$ and $q_1,\ldots,q_s$ are in $Q$. There is a notion of free properad (see \cite[Sect. 2.7]{Val07}), so we can talk about properads presented by generators and relations. As for operads, we have the notion of coloured properads (for the definition of coloured prop, the  reader can refer to \cite{JY09}). The coloured properad $\mathcal{DL}ie\mathcal{R}in$ which encodes the double Lie-Rinehart structure (cf. definition \ref{def::dLieRinehart}), is generated by
	\[
	\begin{tikzpicture}[scale=0.2,baseline=1]
	\draw[thin] (0,-0.5) -- (0,2);
	\draw[thin] (2,-0.5) -- (2,2);
	\draw[fill=white] (-0.3,0.5) rectangle (2.3,1);
	\node  at (0,-0.5) [draw,scale=.25,diamond,fill=black] {};
	\node  at (2,-0.5) [draw,scale=.25,diamond,fill=white] {};
	\node  at (0,2) [draw,scale=.25,diamond,fill=black] {};
	\node  at (2,2) [draw,scale=.25,diamond,fill=black] {};
	\draw (-0.3,0.75) node[left] {$\scriptscriptstyle{f}$};
	\end{tikzpicture}~~ \otimes k
	\quad \oplus \quad
	\begin{tikzpicture}[scale=0.2,baseline=1]
	\draw[thin] (0,-0.5) -- (0,2);
	\draw[thin] (2,-0.5) -- (2,2);
	\draw[fill=lightgray!50!white] (-0.3,0.5) rectangle (2.3,1);
	\node  at (0,-0.5) [draw,scale=.25,diamond,fill=white] {};
	\node  at (2,-0.5) [draw,scale=.25,diamond,fill=white] {};
	\node  at (0,2) [draw,scale=.25,diamond,fill=black] {};
	\node  at (2,2) [draw,scale=.25,diamond,fill=white] {};
	\draw (-0.3,0.75) node[left] {$\scriptscriptstyle{\rho}$};
	\end{tikzpicture}~~\otimes k
	\quad \oplus \quad
	\begin{tikzpicture}[baseline=6,scale=0.2]
	\draw (0.5,3.5) -- (2,2) -- (2,0.5); 
	\draw (3.5,3.5) -- (2,2);
	\node  at (0.5,3.5) [draw,scale=.25,diamond,fill=white] {};
	\node  at (3.5,3.5) [draw,scale=.25,diamond,fill=white] {};
	\node  at (2,0.5) [draw,scale=.25,diamond,fill=white] {};
	\draw[fill=white] (2,2) circle (6pt);
	\draw (2,2) node[below left] {$\scriptscriptstyle{\mu}$};
	\end{tikzpicture}\otimes k[\mathfrak{S}_2]
	\quad \oplus \quad
	\begin{tikzpicture}[baseline=6,scale=0.2]
	\draw (0.5,3.5) -- (2,2) -- (2,0.5); 
	\draw (3.5,3.5) -- (2,2);
	\node  at (0.5,3.5) [draw,scale=.25,diamond,fill=white] {};
	\node  at (3.5,3.5) [draw,scale=.25,diamond,fill=black] {};
	\node  at (2,0.5) [draw,scale=.25,diamond,fill=black] {};
	\draw[fill=white] (2,2) circle (6pt);
	\draw (2,2) node[below left] {$\scriptscriptstyle{l}$};
	\end{tikzpicture}\otimes k
	\quad \oplus \quad
	\begin{tikzpicture}[baseline=6,scale=0.2]
	\draw (0.5,3.5) -- (2,2) -- (2,0.5); 
	\draw (3.5,3.5) -- (2,2);
	\node  at (0.5,3.5) [draw,scale=.25,diamond,fill=black] {};
	\node  at (3.5,3.5) [draw,scale=.25,diamond,fill=white] {};
	\node  at (2,0.5) [draw,scale=.25,diamond,fill=black] {};
	\draw[fill=white] (2,2) circle (6pt);
	\draw (2,2) node[below left] {$\scriptscriptstyle{r}$};
	\end{tikzpicture}\otimes k
	\]
	with the associative relation for the generator $\mu$, and the following relations:
	\begin{equation}\tag{\ref{def::anchor} - derivation}
	\begin{tikzpicture}[scale=0.2,baseline=-1]
	\draw (0,-0.5) -- (0,1);
	\draw (2,-0.5) -- (2,1.5);
	\draw (2,1.5) -- (1,2.5);
	\draw (2,1.5) -- (3,2.5);
	\draw (0,1) to[out=90,in=270] (-1,2.5);
	\node  at (-1,2.5) [draw,scale=.25,diamond,fill=black] {};
	\node  at (1,2.5) [draw,scale=.25,diamond,fill=white] {};
	\node  at (3,2.5) [draw,scale=.25,diamond,fill=white] {};
	\node  at (0,-0.5) [draw,scale=.25,diamond,fill=white] {};
	\node  at (2,-0.5) [draw,scale=.25,diamond,fill=white] {};
	\draw (-1,2.5) node[above] {$\scriptscriptstyle{1}$};
	\draw (1,2.5) node[above] {$\scriptscriptstyle{2}$};
	\draw (3,2.5) node[above] {$\scriptscriptstyle{3}$};
	\draw (0,-0.5) node[below] {$\scriptscriptstyle{1}$};
	\draw (2,-0.5) node[below] {$\scriptscriptstyle{2}$};
	\node  at (2,1.5) [draw,scale=.25,circle,fill=white] {};
	\draw[fill=lightgray!50!white] (-0.3,0.5) rectangle (2.3,1);
	\end{tikzpicture}
	\quad = \quad
	\begin{tikzpicture}[scale=0.2,baseline=-1]
	\draw (1,-1) -- (1,0.5) -- (2,1.5) -- (2,3);
	\draw (1,0.5) -- (0,1.5) -- (0,3);
	\draw (3,-1) to[out=90,in=270] (4,1.5) -- (4,3);
	\node  at (0,3) [draw,scale=.25,diamond,fill=white] {};
	\node  at (2,3) [draw,scale=.25,diamond,fill=black] {};
	\node  at (4,3) [draw,scale=.25,diamond,fill=white] {};
	\node  at (1,-1) [draw,scale=.25,diamond,fill=white] {};
	\node  at (3,-1) [draw,scale=.25,diamond,fill=white] {};
	\draw (0,3) node[above] {$\scriptscriptstyle{2}$};
	\draw (2,3) node[above] {$\scriptscriptstyle{1}$};
	\draw (4,3) node[above] {$\scriptscriptstyle{3}$};
	\draw (1,-1) node[below] {$\scriptscriptstyle{1}$};
	\draw (3,-1) node[below] {$\scriptscriptstyle{2}$};
	\node  at (1,0.5) [draw,scale=.25,circle,fill=white] {};
	\draw[fill=lightgray!50!white] (1.7,1.5) rectangle (4.3,2);
	\end{tikzpicture}
	\quad + \quad
	\begin{tikzpicture}[scale=0.2,baseline=-1]
	\draw (1,0.5) -- (2,1.5) -- (2,3);
	\draw (1,-1) -- (1,0.5) -- (0,1.5) -- (0,3);
	\draw (-1,-1) to[out=90,in=270] (-2,1.5) -- (-2,3);
	\node  at (-2,3) [draw,scale=.25,diamond,fill=black] {};
	\node  at (0,3) [draw,scale=.25,diamond,fill=white] {};
	\node  at (2,3) [draw,scale=.25,diamond,fill=white] {};
	\node  at (1,-1) [draw,scale=.25,diamond,fill=white] {};
	\node  at (-1,-1) [draw,scale=.25,diamond,fill=white] {};
	\draw (-2,3) node[above] {$\scriptscriptstyle{1}$};
	\draw (0,3) node[above] {$\scriptscriptstyle{2}$};
	\draw (2,3) node[above] {$\scriptscriptstyle{3}$};
	\draw (-1,-1) node[below] {$\scriptscriptstyle{1}$};
	\draw (1,-1) node[below] {$\scriptscriptstyle{2}$};
	\node  at (1,0.5) [draw,scale=.25,circle,fill=white] {};
	\draw[fill=lightgray!50!white] (-2.3,1.5) rectangle (0.3,2);
	\end{tikzpicture}~~;
	\end{equation}
	
	\begin{equation}\tag{\ref{def::anchor} - bimodule}
	\begin{tikzpicture}[scale=0.2,baseline=1]
	\draw (0,-0.5) -- (0,2) -- (-1,3) -- (0,2) -- (1,3);
	\draw (2,-0.5) -- (2,1) to[out=90,in=270] (3,2.9);
	\node  at (0,2) [draw,scale=.25,circle,fill=white] {};
	\draw[fill=lightgray!50!white] (-0.3,0.5) rectangle (2.3,1);
	\node  at (0,-0.5) [draw,scale=.25,diamond,fill=white] {};
	\node  at (2,-0.5) [draw,scale=.25,diamond,fill=white] {};
	\node  at (-1,3) [draw,scale=.25,diamond,fill=white] {};
	\node  at (1,3) [draw,scale=.25,diamond,fill=black] {};
	\node  at (3,3) [draw,scale=.25,diamond,fill=white] {};
	\draw (-0.3,0.75) node[left] {$\scriptscriptstyle{\rho}$};
	\draw (0,2) node[left] {$\scriptscriptstyle{l~}$};
	\draw (-1,3) node[above] {$\scriptscriptstyle{1}$};
	\draw (1,3) node[above] {$\scriptscriptstyle{2}$};
	\draw (3,3) node[above] {$\scriptscriptstyle{3}$};
	\draw (0,-0.5) node[below] {$\scriptscriptstyle{1}$};
	\draw (2,-0.5) node[below] {$\scriptscriptstyle{2}$};
	\end{tikzpicture}
	\quad = \quad
	\begin{tikzpicture}[scale=0.2,baseline=1]
	\draw (1,-0.5) to[out=90,in=270] (2,1.5) -- (2,3); 
	\draw (3,0.5) to[out=130,in=270] (0,1.5) -- (0,3);
	\draw (3,-0.5) -- (3,0.5) -- (4,1.5) -- (4,3);
	\node  at (3,0.5) [draw,scale=.25,circle,fill=white] {};
	\draw[fill=lightgray!50!white] (1.7,1.5) rectangle (4.3,2);
	\node  at (1,-0.5) [draw,scale=.25,diamond,fill=white] {};
	\node  at (3,-0.5) [draw,scale=.25,diamond,fill=white] {};
	\node  at (0,3) [draw,scale=.25,diamond,fill=white] {};
	\node  at (2,3) [draw,scale=.25,diamond,fill=black] {};
	\node  at (4,3) [draw,scale=.25,diamond,fill=white] {};
	\draw (0,3) node[above] {$\scriptscriptstyle{1}$};
	\draw (2,3) node[above] {$\scriptscriptstyle{2}$};
	\draw (4,3) node[above] {$\scriptscriptstyle{3}$};
	\draw (1,-0.5) node[below] {$\scriptscriptstyle{1}$};
	\draw (3,-0.5) node[below] {$\scriptscriptstyle{2}$};
	\draw (4.3,1.75) node[right] {$\scriptscriptstyle{\rho}$};
	\draw (3,0.5) node[right] {$\scriptscriptstyle{~\mu}$};
	\end{tikzpicture}
	\qquad ; \qquad
	\begin{tikzpicture}[scale=0.2,baseline=1]
	\draw (0,-0.5) -- (0,2) -- (-1,3) -- (0,2) -- (1,3);
	\draw (2,-0.5) -- (2,1) to[out=90,in=270] (3,2.9);
	\node  at (0,2) [draw,scale=.25,circle,fill=white] {};
	\draw[fill=lightgray!50!white] (-0.3,0.5) rectangle (2.3,1);
	\node  at (0,-0.5) [draw,scale=.25,diamond,fill=white] {};
	\node  at (2,-0.5) [draw,scale=.25,diamond,fill=white] {};
	\node  at (-1,3) [draw,scale=.25,diamond,fill=black] {};
	\node  at (1,3) [draw,scale=.25,diamond,fill=white] {};
	\node  at (3,3) [draw,scale=.25,diamond,fill=white] {};
	\draw (-0.3,0.75) node[left] {$\scriptscriptstyle{\rho}$};
	\draw (0,2) node[left] {$\scriptscriptstyle{r~}$};
	\draw (-1,3) node[above] {$\scriptscriptstyle{1}$};
	\draw (1,3) node[above] {$\scriptscriptstyle{2}$};
	\draw (3,3) node[above] {$\scriptscriptstyle{3}$};
	\draw (0,-0.5) node[below] {$\scriptscriptstyle{1}$};
	\draw (2,-0.5) node[below] {$\scriptscriptstyle{2}$};
	\end{tikzpicture}
	\quad = \quad
	\begin{tikzpicture}[scale=0.2,baseline=1]
	\draw (5,-0.5) to[out=90,in=270] (4,1.5) -- (4,3);
	\draw (6,1.5) -- (6,3);
	\draw (3,0.5) to[out=90,in=270] (6,1.5) -- (6,3);
	\draw (3,-0.5) -- (3,0.5) -- (2,1.5) -- (2,3);
	\node  at (3,0.5) [draw,scale=.25,circle,fill=white] {};
	\draw[fill=lightgray!50!white] (1.7,1.5) rectangle (4.3,2);
	\node  at (3,-0.5) [draw,scale=.25,diamond,fill=white] {};
	\node  at (5,-0.5) [draw,scale=.25,diamond,fill=white] {};
	\node  at (2,3) [draw,scale=.25,diamond,fill=black] {};
	\node  at (4,3) [draw,scale=.25,diamond,fill=white] {};
	\node  at (6,3) [draw,scale=.25,diamond,fill=white] {};
	\draw (6,3) node[above] {$\scriptscriptstyle{2}$};
	\draw (2,3) node[above] {$\scriptscriptstyle{1}$};
	\draw (4,3) node[above] {$\scriptscriptstyle{3}$};
	\draw (3,-0.5) node[below] {$\scriptscriptstyle{1}$};
	\draw (5,-0.5) node[below] {$\scriptscriptstyle{2}$};
	\draw (1.7,1.75) node[left] {$\scriptscriptstyle{\rho}$};
	\draw (3,0.5) node[left] {$\scriptscriptstyle{\mu~}$};
	\end{tikzpicture}~~;
	\end{equation}
	
	\begin{equation}
	\tag{right derivation}
	\begin{tikzpicture}[scale=0.2,baseline=1]
	\draw[thin] (0,-0.5) -- (0,1) to[out=90,in=270] (-1,2.9);
	\draw[thin] (2,-0.5) -- (2,2) -- (3,3);
	\draw (2,2) -- (1,3);
	\node  at (2,2) [draw,scale=.25,circle,fill=white] {};
	\draw[fill=white] (-0.3,0.5) rectangle (2.3,1);
	\node  at (0,-0.5) [draw,scale=.25,diamond,fill=black] {};
	\node  at (2,-0.5) [draw,scale=.25,diamond,fill=white] {};
	\node  at (-1,3) [draw,scale=.25,diamond,fill=black] {};
	\node  at (1,3) [draw,scale=.25,diamond,fill=white] {};
	\node  at (3,3) [draw,scale=.25,diamond,fill=black] {};
	\draw (-0.3,0.75) node[left] {$\scriptscriptstyle{f}$};
	\draw (2,2) node[right] {$\scriptscriptstyle{~l}$};
	\draw (-1,3) node[above] {$\scriptscriptstyle{1}$};
	\draw (1,3) node[above] {$\scriptscriptstyle{2}$};
	\draw (3,3) node[above] {$\scriptscriptstyle{3}$};
	\draw (0,-0.5) node[below] {$\scriptscriptstyle{1}$};
	\draw (2,-0.5) node[below] {$\scriptscriptstyle{2}$};
	\end{tikzpicture}
	\quad = \quad
	\begin{tikzpicture}[scale=0.2,baseline=1]
	\draw (1,-0.5) -- (1,0.5) -- (2,1.5);
	\draw (1,0.5) -- (0,1.5) -- (0,3);
	\draw (2,1.5) -- (2,3);
	\draw (4,2) -- (4,3);
	\draw (3,-0.5) to[out=90,in=270] (4,1.5) ;
	\node  at (1,0.5) [draw,scale=.25,circle,fill=white] {};
	\draw[fill=white] (1.7,1.5) rectangle (4.3,2);
	\node  at (1,-0.5) [draw,scale=.25,diamond,fill=black] {};
	\node  at (3,-0.5) [draw,scale=.25,diamond,fill=white] {};
	\node  at (0,3) [draw,scale=.25,diamond,fill=white] {};
	\node  at (2,3) [draw,scale=.25,diamond,fill=black] {};
	\node  at (4,3) [draw,scale=.25,diamond,fill=black] {};
	\draw (0,3) node[above] {$\scriptscriptstyle{2}$};
	\draw (2,3) node[above] {$\scriptscriptstyle{1}$};
	\draw (4,3) node[above] {$\scriptscriptstyle{3}$};
	\draw (1,-0.5) node[below] {$\scriptscriptstyle{1}$};
	\draw (3,-0.5) node[below] {$\scriptscriptstyle{2}$};
	\draw (4.3,1.75) node[right] {$\scriptscriptstyle{f}$};
	\draw (1,0.5) node[left] {$\scriptscriptstyle{l~}$};
	\end{tikzpicture}
	\qquad ; \qquad
	\begin{tikzpicture}[scale=0.2,baseline=1]
	\draw[thin] (0,-0.5) -- (0,1) to[out=90,in=270] (-1,2.9);
	\draw[thin] (2,-0.5) -- (2,2) -- (3,3);
	\draw (2,2) -- (1,3);
	\node  at (2,2) [draw,scale=.25,circle,fill=white] {};
	\draw[fill=white] (-0.3,0.5) rectangle (2.3,1);
	\node  at (0,-0.5) [draw,scale=.25,diamond,fill=black] {};
	\node  at (2,-0.5) [draw,scale=.25,diamond,fill=white] {};
	\node  at (-1,3) [draw,scale=.25,diamond,fill=black] {};
	\node  at (1,3) [draw,scale=.25,diamond,fill=black] {};
	\node  at (3,3) [draw,scale=.25,diamond,fill=white] {};
	\draw (-0.3,0.75) node[left] {$\scriptscriptstyle{f}$};
	\draw (2,2) node[right] {$\scriptscriptstyle{~r}$};
	\draw (-1,3) node[above] {$\scriptscriptstyle{1}$};
	\draw (1,3) node[above] {$\scriptscriptstyle{2}$};
	\draw (3,3) node[above] {$\scriptscriptstyle{3}$};
	\draw (0,-0.5) node[below] {$\scriptscriptstyle{1}$};
	\draw (2,-0.5) node[below] {$\scriptscriptstyle{2}$};
	\end{tikzpicture}
	\quad = \quad
	\begin{tikzpicture}[scale=0.2,baseline=1]
	\draw (5,-0.5) -- (5,0.5) -- (6,1.5);
	\draw (5,0.5) -- (4,1.5) -- (4,3);
	\draw (2,1.5) -- (2,3);
	\draw (6,1.5) -- (6,3);
	\draw (3,-0.5) to[out=90,in=270] (2,1.5) ;
	\node  at (5,0.5) [draw,scale=.25,circle,fill=white] {};
	\draw[fill=white] (1.7,1.5) rectangle (4.3,2);
	\node  at (3,-0.5) [draw,scale=.25,diamond,fill=black] {};
	\node  at (5,-0.5) [draw,scale=.25,diamond,fill=white] {};
	\node  at (2,3) [draw,scale=.25,diamond,fill=black] {};
	\node  at (4,3) [draw,scale=.25,diamond,fill=black] {};
	\node  at (6,3) [draw,scale=.25,diamond,fill=white] {};
	\draw (6,3) node[above] {$\scriptscriptstyle{3}$};
	\draw (2,3) node[above] {$\scriptscriptstyle{1}$};
	\draw (4,3) node[above] {$\scriptscriptstyle{2}$};
	\draw (3,-0.5) node[below] {$\scriptscriptstyle{1}$};
	\draw (5,-0.5) node[below] {$\scriptscriptstyle{2}$};
	\draw (1.7,1.75) node[left] {$\scriptscriptstyle{f}$};
	\draw (5,0.5) node[right] {$\scriptscriptstyle{~\mu}$};
	\end{tikzpicture}
	\quad + \quad
	\begin{tikzpicture}[scale=0.2,baseline=1]
	\draw (1,-0.5) -- (1,0.5) -- (2,1.5);
	\draw (1,0.5) -- (0,1.5) -- (0,3);
	\draw (2,1.5) -- (2,3);
	\draw (4,2) -- (4,3);
	\draw (3,-0.5) to[out=90,in=270] (4,1.5) ;
	\node  at (1,0.5) [draw,scale=.25,circle,fill=white] {};
	\draw[fill=lightgray!50!white] (1.7,1.5) rectangle (4.3,2);
	\node  at (1,-0.5) [draw,scale=.25,diamond,fill=black] {};
	\node  at (3,-0.5) [draw,scale=.25,diamond,fill=white] {};
	\node  at (0,3) [draw,scale=.25,diamond,fill=black] {};
	\node  at (2,3) [draw,scale=.25,diamond,fill=black] {};
	\node  at (4,3) [draw,scale=.25,diamond,fill=white] {};
	\draw (0,3) node[above] {$\scriptscriptstyle{2}$};
	\draw (2,3) node[above] {$\scriptscriptstyle{1}$};
	\draw (4,3) node[above] {$\scriptscriptstyle{3}$};
	\draw (1,-0.5) node[below] {$\scriptscriptstyle{1}$};
	\draw (3,-0.5) node[below] {$\scriptscriptstyle{2}$};
	\draw (4.3,1.75) node[right] {$\scriptscriptstyle{\rho}$};
	\draw (1,0.5) node[left] {$\scriptscriptstyle{r~}$};
	\end{tikzpicture}~~;
	\end{equation}
	\begin{equation}\tag{bimodule}
	\begin{tikzpicture}[scale=0.2,baseline=1]
	\draw (0,-0.5) -- (0,2) -- (-1,3) -- (0,2) -- (1,3);
	\draw (2,-0.5) -- (2,1) to[out=90,in=270] (3,2.9);
	\node  at (0,2) [draw,scale=.25,circle,fill=white] {};
	\draw[fill=white] (-0.3,0.5) rectangle (2.3,1);
	\node  at (0,-0.5) [draw,scale=.25,diamond,fill=black] {};
	\node  at (2,-0.5) [draw,scale=.25,diamond,fill=white] {};
	\node  at (-1,3) [draw,scale=.25,diamond,fill=white] {};
	\node  at (1,3) [draw,scale=.25,diamond,fill=black] {};
	\node  at (3,3) [draw,scale=.25,diamond,fill=black] {};
	\draw (-0.3,0.75) node[left] {$\scriptscriptstyle{f}$};
	\draw (0,2) node[left] {$\scriptscriptstyle{l~}$};
	\draw (-1,3) node[above] {$\scriptscriptstyle{1}$};
	\draw (1,3) node[above] {$\scriptscriptstyle{2}$};
	\draw (3,3) node[above] {$\scriptscriptstyle{3}$};
	\draw (0,-0.5) node[below] {$\scriptscriptstyle{1}$};
	\draw (2,-0.5) node[below] {$\scriptscriptstyle{2}$};
	\end{tikzpicture}
	\quad = \quad
	\begin{tikzpicture}[scale=0.2,baseline=1]
	\draw (1,-0.5) to[out=90,in=270] (2,1.5) -- (2,3); 
	\draw (3,0.5) to[out=130,in=270] (0,1.5) -- (0,3);
	\draw (3,-0.5) -- (3,0.5) -- (4,1.5) -- (4,3);
	\node  at (3,0.5) [draw,scale=.25,circle,fill=white] {};
	\draw[fill=white] (1.7,1.5) rectangle (4.3,2);
	\node  at (1,-0.5) [draw,scale=.25,diamond,fill=black] {};
	\node  at (3,-0.5) [draw,scale=.25,diamond,fill=white] {};
	\node  at (0,3) [draw,scale=.25,diamond,fill=white] {};
	\node  at (2,3) [draw,scale=.25,diamond,fill=black] {};
	\node  at (4,3) [draw,scale=.25,diamond,fill=black] {};
	\draw (0,3) node[above] {$\scriptscriptstyle{1}$};
	\draw (2,3) node[above] {$\scriptscriptstyle{2}$};
	\draw (4,3) node[above] {$\scriptscriptstyle{3}$};
	\draw (1,-0.5) node[below] {$\scriptscriptstyle{1}$};
	\draw (3,-0.5) node[below] {$\scriptscriptstyle{2}$};
	\draw (4.3,1.75) node[right] {$\scriptscriptstyle{f}$};
	\draw (3,0.5) node[right] {$\scriptscriptstyle{~\mu}$};
	\end{tikzpicture}
	\qquad ; \qquad
	\begin{tikzpicture}[scale=0.2,baseline=1]
	\draw (0,-0.5) -- (0,2) -- (-1,3) -- (0,2) -- (1,3);
	\draw (2,-0.5) -- (2,1) to[out=90,in=270] (3,2.9);
	\node  at (0,2) [draw,scale=.25,circle,fill=white] {};
	\draw[fill=white] (-0.3,0.5) rectangle (2.3,1);
	\node  at (0,-0.5) [draw,scale=.25,diamond,fill=black] {};
	\node  at (2,-0.5) [draw,scale=.25,diamond,fill=white] {};
	\node  at (-1,3) [draw,scale=.25,diamond,fill=black] {};
	\node  at (1,3) [draw,scale=.25,diamond,fill=white] {};
	\node  at (3,3) [draw,scale=.25,diamond,fill=black] {};
	\draw (-0.3,0.75) node[left] {$\scriptscriptstyle{f}$};
	\draw (0,2) node[left] {$\scriptscriptstyle{r~}$};
	\draw (-1,3) node[above] {$\scriptscriptstyle{1}$};
	\draw (1,3) node[above] {$\scriptscriptstyle{2}$};
	\draw (3,3) node[above] {$\scriptscriptstyle{3}$};
	\draw (0,-0.5) node[below] {$\scriptscriptstyle{1}$};
	\draw (2,-0.5) node[below] {$\scriptscriptstyle{2}$};
	\end{tikzpicture}
	\quad = \quad
	\begin{tikzpicture}[scale=0.2,baseline=1]
	\draw (5,-0.5) to[out=90,in=270] (4,1.5) -- (4,3);
	\draw (6,1.5) -- (6,3);
	\draw (3,0.5) to[out=90,in=270] (6,1.5) -- (6,3);
	\draw (3,-0.5) -- (3,0.5) -- (2,1.5) -- (2,3);
	\node  at (3,0.5) [draw,scale=.25,circle,fill=white] {};
	\draw[fill=white] (1.7,1.5) rectangle (4.3,2);
	\node  at (3,-0.5) [draw,scale=.25,diamond,fill=black] {};
	\node  at (5,-0.5) [draw,scale=.25,diamond,fill=white] {};
	\node  at (2,3) [draw,scale=.25,diamond,fill=black] {};
	\node  at (4,3) [draw,scale=.25,diamond,fill=black] {};
	\node  at (6,3) [draw,scale=.25,diamond,fill=white] {};
	\draw (6,3) node[above] {$\scriptscriptstyle{2}$};
	\draw (2,3) node[above] {$\scriptscriptstyle{1}$};
	\draw (4,3) node[above] {$\scriptscriptstyle{3}$};
	\draw (3,-0.5) node[below] {$\scriptscriptstyle{1}$};
	\draw (5,-0.5) node[below] {$\scriptscriptstyle{2}$};
	\draw (1.7,1.75) node[left] {$\scriptscriptstyle{f}$};
	\draw (3,0.5) node[left] {$\scriptscriptstyle{r~}$};
	\end{tikzpicture}~~;
	\end{equation}
	\begin{equation}\tag{\ref{def::DRL_compatible2}}
	\begin{tikzpicture}[scale=0.2,baseline=-1]
	\draw (4,-1) -- (4,0.5) to[out=90,in=270] (2,1.5) -- (2,3);
	\draw (2,-1) -- (2,0) to[out=90,in=270] (4,2) -- (4,3);
	\draw (0,-1) -- (0,3);
	\node  at (0,3) [draw,scale=.25,diamond,fill=black] {};
	\node  at (2,3) [draw,scale=.25,diamond,fill=black] {};
	\node  at (4,3) [draw,scale=.25,diamond,fill=white] {};
	\node  at (0,-1) [draw,scale=.25,diamond,fill=white] {};
	\node  at (2,-1) [draw,scale=.25,diamond,fill=white] {};
	\node  at (4,-1) [draw,scale=.25,diamond,fill=white] {};
	\draw (0,3) node[above] {$\scriptscriptstyle{2}$};
	\draw (2,3) node[above] {$\scriptscriptstyle{3}$};
	\draw (4,3) node[above] {$\scriptscriptstyle{1}$};
	\draw (0,-1) node[below] {$\scriptscriptstyle{3}$};
	\draw (2,-1) node[below] {$\scriptscriptstyle{2}$};
	\draw (4,-1) node[below] {$\scriptscriptstyle{1}$};
	\draw[fill=lightgray!50!white] (1.7,1.5) rectangle (4.3,2);
	\draw (4.3,1.75) node[right] {$\scriptscriptstyle{\rho}$};
	\draw[fill=lightgray!50!white] (-0.3,0) rectangle (2.3,0.5);
	\draw (-0.3,0.25) node[left] {$\scriptscriptstyle{\rho}$};
	\end{tikzpicture}
	\quad - \quad 
	\begin{tikzpicture}[scale=0.2,baseline=-1]
	\draw (0,-1) -- (0,0.5) to[out=90,in=270] (2,1.5) -- (2,3);
	\draw (2,-1) -- (2,0) to[out=90,in=270] (0,2) -- (0,3);
	\draw (4,-1) -- (4,3);
	\node  at (0,3) [draw,scale=.25,diamond,fill=white] {};
	\node  at (2,3) [draw,scale=.25,diamond,fill=black] {};
	\node  at (4,3) [draw,scale=.25,diamond,fill=black] {};
	\node  at (0,-1) [draw,scale=.25,diamond,fill=white] {};
	\node  at (2,-1) [draw,scale=.25,diamond,fill=white] {};
	\node  at (4,-1) [draw,scale=.25,diamond,fill=white] {};
	\draw (0,3) node[above] {$\scriptscriptstyle{1}$};
	\draw (2,3) node[above] {$\scriptscriptstyle{2}$};
	\draw (4,3) node[above] {$\scriptscriptstyle{3}$};
	\draw (0,-1) node[below] {$\scriptscriptstyle{2}$};
	\draw (2,-1) node[below] {$\scriptscriptstyle{1}$};
	\draw (4,-1) node[below] {$\scriptscriptstyle{3}$};
	\draw[fill=white] (1.7,1.5) rectangle (4.3,2);
	\draw (4.3,1.75) node[right] {$\scriptscriptstyle{f}$};
	\draw[fill=lightgray!50!white] (-0.3,0) rectangle (2.3,0.5);
	\draw (-0.3,0.25) node[left] {$\scriptscriptstyle{\rho}$};
	\end{tikzpicture}
	\quad - \quad
	\begin{tikzpicture}[scale=0.2,baseline=-1]
	\draw (4,-1) -- (4,0.5) to[out=90,in=270] (2,1.5) -- (2,3);
	\draw (2,-1) -- (2,0.3) to[out=90,in=270] (4,1.8) -- (4,3);
	\draw (0,-1) -- (0,3);
	\node  at (0,3) [draw,scale=.25,diamond,fill=black] {};
	\node  at (2,3) [draw,scale=.25,diamond,fill=black] {};
	\node  at (4,3) [draw,scale=.25,diamond,fill=white] {};
	\node  at (0,-1) [draw,scale=.25,diamond,fill=white] {};
	\node  at (2,-1) [draw,scale=.25,diamond,fill=white] {};
	\node  at (4,-1) [draw,scale=.25,diamond,fill=white] {};
	\draw (0,3) node[above] {$\scriptscriptstyle{3}$};
	\draw (2,3) node[above] {$\scriptscriptstyle{2}$};
	\draw (4,3) node[above] {$\scriptscriptstyle{1}$};
	\draw (0,-1) node[below] {$\scriptscriptstyle{3}$};
	\draw (2,-1) node[below] {$\scriptscriptstyle{1}$};
	\draw (4,-1) node[below] {$\scriptscriptstyle{2}$};
	\draw[fill=lightgray!50!white] (1.7,1.5) rectangle (4.3,2);
	\draw (4.3,1.75) node[right] {$\scriptscriptstyle{\rho}$};
	\draw[fill=lightgray!50!white] (-0.3,0) rectangle (2.3,0.5);
	\draw (-0.3,0.25) node[left] {$\scriptscriptstyle{\rho}$};
	\end{tikzpicture}
	\quad = \quad 0~~;
	\end{equation}
	\begin{equation}\tag{\ref{def::DRL_DJacobi}}
	\begin{tikzpicture}[scale=0.2,baseline=-1]
	\draw[thin] (0,-1) -- (0,3);
	\draw[thin] (2,-1) -- (2,3);
	\draw[thin] (4,-1) -- (4,3);
	\node  at (0,3) [draw,scale=.25,diamond,fill=black] {};
	\node  at (2,3) [draw,scale=.25,diamond,fill=black] {};
	\node  at (4,3) [draw,scale=.25,diamond,fill=black] {};
	\node  at (0,-1) [draw,scale=.25,diamond,fill=black] {};
	\node  at (2,-1) [draw,scale=.25,diamond,fill=white] {};
	\node  at (4,-1) [draw,scale=.25,diamond,fill=white] {};
	\draw (0,3) node[above] {$\scriptscriptstyle{1}$};
	\draw (2,3) node[above] {$\scriptscriptstyle{2}$};
	\draw (4,3) node[above] {$\scriptscriptstyle{3}$};
	\draw (0,-1) node[below] {$\scriptscriptstyle{1}$};
	\draw (2,-1) node[below] {$\scriptscriptstyle{2}$};
	\draw (4,-1) node[below] {$\scriptscriptstyle{3}$};
	\draw[fill=white] (1.7,1.5) rectangle (4.3,2);
	\draw (4.3,1.75) node[right] {$\scriptscriptstyle{f}$};
	\draw[fill=white] (-0.3,0) rectangle (2.3,0.5);
	\draw (-0.3,0.25) node[left] {$\scriptscriptstyle{f}$};
	\end{tikzpicture}
	\quad - \quad 
	\begin{tikzpicture}[scale=0.2,baseline=-1]
	\draw (0,-1) -- (0,0.5) to[out=90,in=270] (2,1.5) -- (2,3);
	\draw (2,-1) -- (2,0) to[out=90,in=270] (0,2) -- (0,3);
	\draw (4,-1) -- (4,3);
	\node  at (0,3) [draw,scale=.25,diamond,fill=black] {};
	\node  at (2,3) [draw,scale=.25,diamond,fill=black] {};
	\node  at (4,3) [draw,scale=.25,diamond,fill=black] {};
	\node  at (0,-1) [draw,scale=.25,diamond,fill=black] {};
	\node  at (2,-1) [draw,scale=.25,diamond,fill=white] {};
	\node  at (4,-1) [draw,scale=.25,diamond,fill=white] {};
	\draw (0,3) node[above] {$\scriptscriptstyle{3}$};
	\draw (2,3) node[above] {$\scriptscriptstyle{1}$};
	\draw (4,3) node[above] {$\scriptscriptstyle{2}$};
	\draw (0,-1) node[below] {$\scriptscriptstyle{1}$};
	\draw (2,-1) node[below] {$\scriptscriptstyle{3}$};
	\draw (4,-1) node[below] {$\scriptscriptstyle{2}$};
	\draw[fill=white] (1.7,1.5) rectangle (4.3,2);
	\draw (4.3,1.75) node[right] {$\scriptscriptstyle{f}$};
	\draw[fill=white] (-0.3,0) rectangle (2.3,0.5);
	\draw (-0.3,0.25) node[left] {$\scriptscriptstyle{f}$};
	\end{tikzpicture}
	\quad - \quad 
	\begin{tikzpicture}[scale=0.2,baseline=-1]
	\draw (4,-1) -- (4,0.5) to[out=90,in=270] (2,1.5) -- (2,3);
	\draw (2,-1) -- (2,0.3) to[out=90,in=270] (4,1.8) -- (4,3);
	\draw (0,-1) -- (0,3);
	\node  at (0,3) [draw,scale=.25,diamond,fill=black] {};
	\node  at (2,3) [draw,scale=.25,diamond,fill=black] {};
	\node  at (4,3) [draw,scale=.25,diamond,fill=black] {};
	\node  at (0,-1) [draw,scale=.25,diamond,fill=white] {};
	\node  at (2,-1) [draw,scale=.25,diamond,fill=white] {};
	\node  at (4,-1) [draw,scale=.25,diamond,fill=black] {};
	\draw (0,3) node[above] {$\scriptscriptstyle{2}$};
	\draw (2,3) node[above] {$\scriptscriptstyle{1}$};
	\draw (4,3) node[above] {$\scriptscriptstyle{3}$};
	\draw (0,-1) node[below] {$\scriptscriptstyle{2}$};
	\draw (2,-1) node[below] {$\scriptscriptstyle{3}$};
	\draw (4,-1) node[below] {$\scriptscriptstyle{1}$};
	\draw[fill=white] (1.7,1.5) rectangle (4.3,2);
	\draw (4.3,1.75) node[right] {$\scriptscriptstyle{f}$};
	\draw[fill=lightgray!50!white] (-0.3,0) rectangle (2.3,0.5);
	\draw (-0.3,0.25) node[left] {$\scriptscriptstyle{\rho}$};
	\end{tikzpicture}
	\quad = \quad 0.
	\end{equation}
\end{rem}

In the next proposition, we establish the noncommutative version of the correspondence between Lie-Rinehart algebras and a class of Poisson algebras stated in \ref{prop::LieRin_Pois}. Namely, we explain the correspondence between $\Sigma$-double Lie-Rinehart algebras and linear $\Sigma$-double Poisson algebras.

\begin{prop}[{cf.\cite[(3.4-1)-(3.4-8)]{VdB08} -- \cite[Sect. 3.2]{VdB08-2}}] \label{prop::eq_LieRinehart_DPlin}
	Let $A$ be a monoid in $\C $, $M$ an $A$-bimodule and $\Sigma$, an invertible object in $\C$. The following assertions are equivalent: 
	\begin{enumerate}
		\item $M$ is a $\Sigma$-double Lie-Rinehart algebra over $A$;
		\item $T_A(M)$ is a linear $\Sigma$-double Poisson algebra.
	\end{enumerate}
	We have the equivalence of categories
	\[
	\Sigma\mbox{-}\mathtt{DPFree}_A^{lin} \cong \Sigma\mbox{-}\mathtt{DLR}_A.
	\]
\end{prop}
\begin{proof}
	The anchor $\rho$ of a $\Sigma$-double Lie-Rinehart algebra $M$ over $A$ is a morphism $\rho:\Sigma M \otimes \Sigma A \rightarrow \Sigma A\otimes A$ which we extend by antisymmetry, to a double bracket $\DB{-}{-}^A : \Sigma M \otimes \Sigma A \oplus \Sigma A \otimes \Sigma M \rightarrow \Sigma A\otimes A$.
	
	The condition \eqref{def::DRL_compatible1} of definition \ref{def::dLieRinehart} corresponds to the derivation properties of the restriction $\Sigma M \otimes \Sigma A \oplus \Sigma A \otimes \Sigma M$ of a linear $\Sigma$-double bracket on $T_A M$. Hence, a linear $\Sigma$-double bracket on $T_A M$ corresponds to an anchor $\rho : \Sigma M \otimes \Sigma A\rightarrow \Sigma A\otimes A$ and a morphism  $f : \Sigma M\otimes \Sigma M \rightarrow \Sigma (M \otimes A \oplus A\otimes M)$ such that conditions \eqref{def::DLR_antisym} and \eqref{def::DRL_compatible1} of definition \ref{def::dLieRinehart} are satisfied. 
	
	We will check that the conditions \eqref{def::DRL_compatible2} and \eqref{def::DRL_DJacobi} of definition \ref{def::dLieRinehart} exactly correspond to the double Jacobi identity for the associated linear $\Sigma$-double bracket on $T_A(M)$. The double-jacobiator of $T_AM$, restricted to $\Sigma A\otimes \Sigma M \otimes \Sigma M$ is given by the following diagram 
	\[
	\xymatrix@C=0.78cm{
		\Sigma M\otimes \Sigma A\otimes \Sigma M \ar[d]|(0.4){(\rho^*\otimes A)(\Sigma M\otimes \rho^*_r)}
		&& \Sigma A\otimes \Sigma M \otimes \Sigma M 
		\ar[d]|(0.4){(\rho^*_r\otimes A)(\Sigma A\otimes \DB{-}{-}_l)}
		\ar[rr]^{\tau_{\Sigma A,\Sigma M \otimes \Sigma M }}
		\ar[ll]_{\tau_{\Sigma A \otimes \Sigma M , \Sigma M }}
		&& \Sigma M \otimes \Sigma M \otimes \Sigma A 
		\ar[d]|(0.4){(\rho^*\otimes A)(\Sigma A\otimes \rho^*)} \\
		\Sigma \big( A\big)^{\otimes 3}\ar[rr]_{\Sigma\tau_{ A, A\otimes A}} 
		&& \Sigma\big(A\big)^{\otimes 3}  
		&& \Sigma\big(A\big)^{\otimes 3} \ar[ll]^{\Sigma\tau_{A\otimes  A,  A}} ,
	}
	\]
	then, morphisms $\rho$ and $\DB{-}{-}$ satisfy the condition \eqref{def::DRL_compatible2} of the $\Sigma$-double Lie-Rinehart algebra. 	The restriction to $(\Sigma M)^{\otimes 3}$ of the double-jacobiator on $T_A(M)$ takes values in
	\[
		\Sigma( M \otimes A^{\otimes 2}\oplus A^{\otimes 2}\otimes M \oplus A\otimes M\otimes A).
	\]
	By invariance under the $\Z/3\Z$-action (see remark \ref{rem::dJacobi_stable_Z3Z}), the vanishing of this restriction is equivalent to the vanishing of its projection on $\Sigma M\otimes A^{\otimes 2}$. This projection is given by the sum of the morphisms $(\Sigma M)^{\otimes 3}\rightarrow \Sigma M\otimes A^{\otimes 2}$ given in  the following commutative diagram
	\[
	\xymatrix@C=1cm{
		\big(\Sigma M\big)^{\otimes 3}
		\ar[d]|{(\DB{-}{-}^M_r\otimes A)(\Sigma M\otimes \DB{-}{-}^M_l)~~}
		&& \big(\Sigma M\big)^{\otimes 3}
		\ar[d]|{(\DB{-}{-}_l^M\otimes A)(\Sigma A\otimes \DB{-}{-}_l^M)}
		\ar[rr]^{\tau_{\Sigma M,\Sigma M \otimes \Sigma M }}
		\ar[ll]_{\tau_{\Sigma M \otimes \Sigma M , \Sigma M }}
		&&\big(\Sigma M\big)^{\otimes 3}
		\ar[d]|{(\rho^M\otimes A)(\Sigma M\otimes\DB{-}{-}^M_r)} \\
		\Sigma A\otimes M \otimes A \ar[rr]_{\Sigma\tau_{ A, M\otimes  A} } 
		&& \Sigma M\otimes A^{\otimes 2}  
		&& \Sigma A^{\otimes 2} \otimes M \ar[ll]^{\Sigma\tau_{ A\otimes  A, M}};
	}
	\]
	the vanishing of the restriction of the double-jacobiator to  $(\Sigma M)^{\otimes 3}$ is equivalent to the condition \eqref{def::DRL_DJacobi}  of definition \ref{def::dLieRinehart}. 
	
	Then, if we consider $(T_A(M),f:=\DB{-}{-})$ a linear $\Sigma$-double Poisson algebra,  by taking the following restrictions of the linear $\Sigma$-double bracket, the morphisms
	\[
		\rho^M:= f|_{\Sigma M \otimes A}~~\mbox{and}~~ \DB{-}{-}^{M}:=f|_{\Sigma M \otimes \Sigma M}
	\]
	make $M$ a $\Sigma$-double Lie Rinehart algebra over $A$.
	
	Conversely, consider $(M,\rho,\DB{-}{-})$ a $\Sigma$-double Lie Rinehart algebra over $A$. By the universal property of $T_A(M)$ (see \ref{sect::DP_free_monoid}), we extend by derivation the morphism $\rho : \Sigma M \otimes A \rightarrow \Sigma A\otimes A$ to 
	\[
		\widetilde{\rho} : \Sigma T_A(M) \otimes A \rightarrow \Sigma T_A(M)\otimes T_A(M),
	\]
	which is a derivation in its first input. We extend $\widetilde{\rho}$ to a morphism 
	\[
		\DB{-}{-}^A : \Sigma T_A(M) \otimes \Sigma T_A(M) \rightarrow \Sigma T_A(M)\otimes T_A(M)
	\]
	by antisymmetry. Similarly, we extend $\DB{-}{-}$ to a double derivation
	\[
		\DB{-}{-}^M : \Sigma T_A(M) \otimes \Sigma T_A(M) \rightarrow \Sigma T_A(M)\otimes T_A(M).
	\]
	The sum $\DB{-}{-}^A+\DB{-}{-}^M$ gives a linear $\Sigma$-double Poisson bracket on $T_A(M)$ because we have proved that the double Lie-Rinehart conditions \eqref{def::DLR_antisym}, \eqref{def::DRL_compatible1}, \eqref{def::DRL_compatible2} and \eqref{def::DRL_DJacobi} are equivalent to the the double Poisson conditions.
	
	Let $(M,\DB{-}{-}^M)$ and $(N,\DB{-}{-}^N)$ be two $\Sigma$-double Lie-Rinehart algebras and $\phi : M \rightarrow N$, a double Lie-Rinehart algebra morphism. The morphism $\phi$ ginduces a morphism of $A$-algebras $\phi' : T_A(M) \rightarrow T_A(N)$. We have $\phi\big(\DB{-}{-}^{M}\big)=\DB{\phi(-)}{\phi(-)}^{N}$ and as $\phi'$ is algebra morphism and the double brackets on $\DB{-}{-}^{M}$ and $\DB{-}{-}^{N}$ are constructed by extension by derivation, then $\phi'$ is a $\Sigma$-double Poisson algebra morphism. Hence, we have defined the functor
	\begin{equation}\label{eq::functor}
	T_A(-): \Sigma\mbox{-}\mathtt{DPFree}_A^{lin} \rightarrow \Sigma\mbox{-}\mathtt{DLR}_A.
	\end{equation}
	Let $\psi: T_A(M) \rightarrow T_A(N)$ be  a morphism of linear $\Sigma$-double Poisson algebras. The morphism $\psi' : pr_N\circ \psi\circ i : M \rightarrow N$ with $i : M \hookrightarrow T_A(M)$ and $pr_N : T_A(N) \twoheadrightarrow N$ is an  $A$-bimodule morphism, which commutes with the double brackets. Then,  the functor \ref{eq::functor} is an equivalence of categories.
\end{proof}

\begin{rem}
	By proposition \ref{double_Poisson_Leibniz}, if $M$ is a  $\mathds{1}$-double Lie-Rinehart algebra over $A$, then the morphism 
	\[
	\{-,-\}:=\mu_A^M\DB{-}{-}^M_l +_A\!\mu^M\DB{-}{-}^M_r : M \otimes  M \longrightarrow M
	\]
	yields on  $M$ a left Leibniz algebra structure, which is an $A$-bimodule morphism, where the $A$-bimodule structure on $M\otimes M$ is given by the $A$-bimodule structure of the right factor.
	
	 The composition of the anchor with the product of $A$ is a derivation in its second input:
	\[
	\widetilde{\rho}:=\mu\circ\rho : M \otimes A \longrightarrow A.
	\]
By proposition \ref{prop::eq_LieRinehart_DPlin} and a generalization of \cite[Prop. 2.4.2]{VdB08}, the double-Jacobi identity, restricted to $M\otimes M\otimes A$ implies that $\widetilde{\rho}$ defines on $A$ the structure of an antisymmetrical representation of $M$ (for the definition of representations of left Leibniz algebras, readers can refer to \cite[Def. 1.2.1 and 1.2.4]{Cov10}).
\end{rem}

\section{The shifting property}\label{subsect::decalage}
\subsection{The main result}
In the case of algebras over an operad, a $\Sigma$-shifted structure on an object $M$ is equivalent to a non-shifted structure on $\Sigma M$ (for more detail, the reader can refer to \cite{LV12}). However, it is not true for the case of algebras over a properad. For example, for a chain complex $A$, an $r$-double Lie structure on $A$ is the datum of a double bracket
\[
\DB{-}{-} : A[r]\otimes A[r] \longrightarrow (A\otimes A)[r]
\]
and a $0$-double Lie structure on $A[r]$ is the datum of a double bracket
\[
\DB{-}{-} : A[r]\otimes A[r] \longrightarrow A[r]\otimes A[r]. 
\]
We directly see that these morphisms have different degrees.
However, this shifting property is true for double Lie-Rinehart algebras (which are algebras over the coloured properad $\mathcal{DL}ie\mathcal{R}in$). In fact, the equivalence of categories $\Sigma\otimes-: A\mbox{-}\mathtt{Bimod}_\C \rightarrow A\mbox{-}\mathtt{Bimod}_\C$ induces an equivalence of categories $\Sigma\mbox{-}\mathtt{DLR}_A \rightarrow \mathds{1}\mbox{-}\mathtt{DLR}_A$:
\begin{thm}\label{prop::equivalence_decalage}
	The following assertions are equivalent:
	\begin{enumerate}
		\item $(A,M,\rho_M,\DB{-}{-}_M)$ is a $\Sigma$-double Lie-Rinehart algebra;
		\item $(A,\Sigma M,\rho_{\Sigma M},\DB{-}{-}_{\Sigma M})$ is a  $\mathds{1}$-double Lie-Rinehart algebra;
	\end{enumerate}
	where the anchors $\rho_M$ and $\rho_{\Sigma M}$ are related by 
	\[
		\rho_{\Sigma M}=(\tau_{\Sigma^-, \Sigma M}\otimes A) (\Sigma^-\otimes\rho_{M})
	\]
	(where we implicitly use the isomorphism $\Sigma^-\otimes \Sigma \cong \mathds{1}$) and the double brackets satisfy 
	\[
	\DB{-}{-}_{\Sigma M} =(\Sigma M\otimes A,\tau_{\Sigma,A}\otimes M)\circ \DB{-}{-}.
	\]
	Under this correspondance, we have the equivalence of categories
	\[
	\Sigma\mbox{-}\mathtt{DLR}_A \cong \mathds{1}\mbox{-}\mathtt{DLR}_A.
	\]
\end{thm}
\begin{proof}
	An $A$-bimodule structure on $M$ canonically corresponds to an $A$-bimodule structure on $\Sigma M$ (see section \ref{subsect::bimodule_structure}). The commutative square
	\[
	\xymatrix@R=0.5cm{
		\Sigma(M\otimes A\oplus A\otimes M) \ar[r]^{\cong} \ar[d]_{\Sigma \tau_{M,A}} &\Sigma M\otimes A\oplus A\otimes\Sigma M \ar[d]^{\tau_{\Sigma M,A}} \\
		\Sigma(M\otimes A\oplus A\otimes M) \ar[r]_{\cong} & \Sigma M\otimes A\oplus A\otimes\Sigma M
	}
	\]
	implies we have a canonical correspondance between an antisymmetrical $\Sigma$-double bracket 	$\DB{-}{-}_M: \Sigma M\otimes \Sigma M \longrightarrow \Sigma (M \otimes A \oplus A\otimes M)$ and an antisymmetrical  $\mathds{1}$-double bracket $\DB{-}{-}_{\Sigma M}: \Sigma M\otimes \Sigma M \longrightarrow \Sigma M \otimes A \oplus A\otimes\Sigma M$ given by 
	\[
	\DB{-}{-}_{\Sigma M} =(\Sigma M\otimes A,\tau_{\Sigma,A}\otimes M)\circ \DB{-}{-}.
	\]
	Futhermore, $\DB{-}{-}_M$ satisfies the double Jacobi identity if and only if $\DB{-}{-}_{\Sigma M}$ does too.
	
	By the remark \ref{rem::equivalence_categorie}, the functor $\Sigma^-\otimes -: A\mbox{-}\mathtt{Bimod}_\C \rightarrow  A\mbox{-}\mathtt{Bimod}_\C$ is an equivalence of categories: the morphism $\rho_{M}: \Sigma M \otimes \Sigma A \rightarrow \Sigma A\otimes A$ corresponds to the morphism $\rho_{\Sigma M}:=(\tau_{\Sigma^-, \Sigma M}\otimes A) (\Sigma^-\otimes\rho_{M}): \Sigma M \otimes A \rightarrow A\otimes A$ (where we implicitly  use the isomorphism $\Sigma^-\otimes \Sigma \cong \mathds{1}$). Conversely, by the equivalence $\Sigma\otimes -: A\mbox{-}\mathtt{Bimod}_\C \rightarrow  A\mbox{-}\mathtt{Bimod}_\C$, a morphism $\rho_{\Sigma M}:\Sigma M \otimes A \rightarrow A\otimes A$ corresponds to $\rho_{M}: \Sigma M \otimes \Sigma A \rightarrow \Sigma A\otimes A$. So $\rho_M$ satisfies the condition \eqref{def::DRL_compatible2} of definition  \ref{def::dLieRinehart} if and only if $\rho_{\Sigma M}$ satisfies the condition \eqref{def::DRL_compatible2} of definition \ref{def::dLieRinehart}. 
	
	The equivalence of the anchors' compatibility (cf. condition \eqref{def::dLieRinehart}\eqref{def::DRL_compatible1}) remains to be done. We need to show that the following diagram:
	\[
	\xymatrix{
		& \Sigma M \otimes  A\otimes\Sigma M \ar[dl]_{\cong} \ar@{..>}@/_1pc/[rrdd]^(0.8){\psi^l} \ar[r]^{\Sigma M \otimes _A\!\mu }& \Sigma M \otimes \Sigma M \ar[dl]_{=} \ar@{..>}@/^1pc/[rdd]|(0.7){\DB{-}{-}_{\Sigma M}} & \Sigma M \otimes \Sigma M\otimes A \ar[l]_{\Sigma M \otimes  \mu_A }\ar[dl]_{=} \ar[dd]^{\psi^r} \\
		\Sigma M \otimes \Sigma (A\otimes M) \ar[r]^{\Sigma M \otimes \Sigma _A\!\mu }  \ar@/_1pc/[rd]_(0.4){\phi^l}
		& \Sigma M \otimes \Sigma M \ar[d]^{\DB{-}{-}_M} 
		& \Sigma M \otimes \Sigma (M\otimes A) \ar[l]_{\Sigma M \otimes \Sigma \mu_A } \ar@/^1pc/[ld]^(0.4){\phi^r} &\\
		& \Sigma( M\otimes A \oplus A\otimes M) && \ar[ll]_{\cong} \Sigma M\otimes A \oplus A\otimes\Sigma M
	}
	\]
	commutes, with
	\begin{align*}
	\psi^l:= & ((_A\!\mu,\mu)\otimes M)(A\otimes \DB{-}{-}_{\Sigma M})(\tau_{\Sigma M, A}\otimes\Sigma M) +(A\otimes _A\!\mu)(\rho_{\Sigma M}\otimes \Sigma M)~; \\
	\psi^r:= & (\Sigma M\otimes \mu + A \otimes \mu_A)( \DB{-}{-}_{\Sigma M} \otimes A) + (\mu_A\otimes A)(\Sigma M\otimes \rho_{\Sigma M})(\tau_{\Sigma M, \Sigma M}\otimes A)~;\\
	\phi^l:= & (_A\!\mu\otimes M)(A\otimes \DB{-}{-}_M)(\tau_{\Sigma M \Sigma, A}\otimes M) +(\Sigma A\otimes _A\!\mu)(\rho_M\otimes M)~;\\
	\phi^r:= & (\Sigma M\otimes \mu + \Sigma A \otimes \mu_A)(\DB{-}{-}_M\otimes A)+ (\Sigma \mu_A\otimes A)(\tau_{M,\Sigma}\otimes A\otimes A)(M\otimes \rho_M)(\tau_{\Sigma M\Sigma,M}\otimes A).
	\end{align*}
	It suffices to show that the following squares
	\[
	\xymatrix{
		\Sigma M \otimes  A\otimes\Sigma M \ar[r]^{\psi^l} \ar[d]^{\cong}_{\Sigma M\otimes \tau_{A,\Sigma}\otimes M} &  \Sigma M\otimes A \oplus A\otimes\Sigma M \ar[d]_{\cong}^{(\mathrm{id},\tau_{A,\Sigma}\otimes M)}\\
		\Sigma M \otimes \Sigma (A\otimes M) \ar[r]_{\phi^l} & \Sigma( M\otimes A \oplus A\otimes M)
	}
	\qquad 
	\xymatrix{
		\Sigma M \otimes  \Sigma M\otimes A \ar[r]^{\psi^r} \ar[d]_{=} &  \Sigma M\otimes A \oplus A\otimes\Sigma M \ar[d]_{\cong}^{(\mathrm{id},\tau_{A,\Sigma}\otimes M)}\\
		\Sigma M \otimes \Sigma (M\otimes A) \ar[r]_{\phi^r} & \Sigma( M\otimes A \oplus A\otimes M)
	}
	\]
	commute. We have commutative diagrams
	\[
	\xymatrix{
		\Sigma M \otimes A \otimes \Sigma M \ar[d]_{\rho_{\Sigma M}\otimes \Sigma M} \ar@<0.2pc>[r]^{\Sigma M \otimes \tau_{A,\Sigma}\otimes M}
		& \Sigma M \otimes \Sigma A \otimes M \ar[d]^{\rho_M\otimes M}~~\\
		A\otimes A \otimes \Sigma M \ar[d]_{A \otimes _A\!\mu}
		& \Sigma A\otimes A \otimes M \ar[d]^{\Sigma A\otimes _A\!\mu}~~\\
		A \otimes \Sigma M \ar[r]_{\tau_{A,\Sigma}\otimes M}
		& \Sigma A\otimes M
	}
	\qquad 
	\xymatrix{
		\Sigma M \otimes \Sigma M \otimes A \ar[r]^{=} \ar[d]_{\tau_{\Sigma M,\Sigma M}\otimes A}
		& \Sigma M \otimes \Sigma M \otimes A \ar[d]^{\tau_{\Sigma M \Sigma,M}\otimes A}~~\\
		\Sigma M \otimes \Sigma M \otimes A \ar[d]_{\Sigma M\otimes \rho_{\Sigma M}} 
		& M\otimes \Sigma M \otimes \Sigma A \ar[d]^{M\otimes \rho_M} ~~\\
		\Sigma M\otimes A\otimes A \ar[d]_{\mu_A\otimes A}
		&M\otimes \Sigma A\otimes A \ar[d]|{(\Sigma\mu_A\otimes A)(\tau_{M,\Sigma}\otimes A\otimes A)}~~\\
		\Sigma M \otimes A \ar[r]_{=} & \Sigma M\otimes A		
	}~~;
	\]
	as $\DB{-}{-}_{\Sigma M} = (\mathrm{id},\tau_{\Sigma,A}\otimes M)\circ\DB{-}{-}_M$, we have following equalities:
	\begin{align*}
	&(\mathrm{id},\tau_{A,\Sigma}\otimes M)((_A\!\mu,\mu)\otimes M)(A\otimes \DB{-}{-}_{\Sigma M})(\tau_{\Sigma M, A}\otimes\Sigma M)=~~\\
	&\qquad(_A\!\mu\otimes M)(A\otimes \DB{-}{-}_M)(\tau_{\Sigma M \Sigma, A}\otimes M)(\Sigma M\otimes \tau_{A,\Sigma}\otimes M)
	\end{align*}
	and
	\[
	(\mathrm{id},\tau_{A,\Sigma}\otimes M)(\Sigma M\otimes \mu + A \otimes \mu_A)( \DB{-}{-}_{\Sigma M} \otimes A) =(\Sigma M\otimes \mu + \Sigma A \otimes \mu_A)(\DB{-}{-}_M\otimes A)
	\]
	so $(\mathrm{id},\tau_{A,\Sigma}\otimes M) \circ \psi^r = \phi^r$. 
\end{proof}

\begin{exmp}
	In the case where $\C$ is the category of chain complexes $\Ch_k$,  we consider a linear $0$-double Poisson bracket on $T(M)$, as in example \ref{exmp::DLR_associatif}, which corresponds to a (non-unital) associative product on $M$,  by example \ref{exmp::DLR_associatif}. Then, after proposition  \ref{prop::eq_LieRinehart_DPlin}, the shifting property implies that $M[1]$ has a (non-unital) associative product of homological degree $-1$. This recovers the shifting property of algebras over the operad $\mathcal{A}s$: a (non-unital) associative product of degree $-1$ on $M$ correspond to a non-unital associative algebra structure of degree $0$ on $M[-1]$.
\end{exmp}


\subsection{Example : the Koszul double bracket}\label{subsect::dcrochet_Koszul}
For this example, we suppose the category $\C$ has coequalizers. We begin by recalling the definition of the $A$-bimodule of noncommutative one forms associated to a unital monoid $A$ (see \cite{Lod98} for details in the case of $\C=\Ch_k$).
\begin{prop}[Noncommutative differential $1$-forms]\label{def::OmegaA}
	Let $(A,\mu	,\iota)$ be a unital associative monoid in $\C$. We define the \emph{$A$-bimodule of noncommutative differential $1$-forms} $\Omega_A$ as the coequalizer 
	\[
	\xymatrix{
		A^{\otimes 4}  \ar@<-0.2pc>[rr]_{\mu\otimes A^{\otimes 2}+ A^{\otimes 2}\otimes \mu} \ar@<0.2pc>[rr]^{A\otimes \mu\otimes A}
		&& A^{\otimes 3} \ar[r]^{\tilde{d}}
		& \Omega_A^1
	}
	\]
	in the category of $A$-bimodules in $\C$, where $A^{\otimes i}$, for $i=3,4$, are equipped with their external $A$-bimodule structure. We denote by 
	\[
	d:\tilde{d}(\iota\otimes A\otimes \iota) : A \longrightarrow \Omega_A
	\]
	the universal derivation.
\end{prop}

\begin{prop}\label{prop::sDer_representable}
	Let $A$ be a unital associative monoid. 	The $A$-bimodule $\Omega_A$ satisfies the following universal property: for an $A$-bimodule $M$ and a derivation $h : A\rightarrow M$, there exists a unique $A$-bimodule morphism $i_h : \Omega_A \rightarrow M$ such that $h=i_h\circ d$. That is, we have the canonical isomorphism 
	\[
	\begin{array}{ccc}
		\Der(A,M)& \overset{\cong}{\longrightarrow} & \Hom_{A\mbox{-}\mathtt{Bimod}}(\Omega_A,M). \\
		h & \longmapsto &i_h
	\end{array}
	\]
	where $\Der(A,M)$ is the subset of $\Hom_\C(A,M)$ of morphisms satisfying the derivation property, i.e. $\phi:A\rightarrow M \in \Der(A,M)$ if $\phi\circ\mu= \mu_A(\phi\otimes A) + _A\!\mu(A\otimes \phi)$.  We also have the canonical isomorphism of $A$-bimodules: for an $A$-bimodule $M$ and $X$ and $Y$ two objects in $\C$
	\[
	\Der(X\otimes A\otimes Y,M)\cong \Hom_{A\mbox{-}\mathtt{Bimod}}(X\otimes \Omega_A \otimes Y,M).
	\]
\end{prop}

In this section, we consider $A$, a unital associative monoid, with a $\Sigma$-double Poisson bracket:
\[
\DB{-}{-}:\Sigma A \otimes \Sigma A \longrightarrow \Sigma A\otimes A.
\]
Canonically, we associate to this double-bracket, a $\Sigma$-double Lie-Rinehart structure (over $A$) on $\Omega_A$, the $A$-bimodule of noncommutative one forms. By the derivation property of $\DB{-}{-}$, proposition \ref{prop::sDer_representable} implies that we can extend the double bracket to the following $A^e$-bimodule morphism:
\[
	\phi : \Sigma \Omega_A \otimes \Sigma \Omega_A  \longrightarrow\Sigma A\otimes A,
\]
where the $A^e$-bimodule structure is given by the external and internal $A$-bimodule structures. By composition with $d$ (extended as a derivation to $A\otimes A$), we obtain, in the category $\C$, the morphism
$\DB{-}{-}^\Omega$, called the \emph{Koszul double bracket}: 
\[
\xymatrix@C=2cm{
	\Sigma A\otimes \Sigma A \ar[r]^{\DB{-}{-}} \ar[d]_{\Sigma d \otimes \Sigma d}& \Sigma A\otimes A \ar[d]^{d\otimes A+A\otimes d} \\
	\Sigma \Omega_A \otimes \Sigma \Omega_A \ar[ru]|{\phi} \ar@{.>}[r]_(0.4){=: \DB{-}{-}^\Omega}& \Sigma\big( \Omega_A\otimes A \oplus A\otimes \Omega_A\big) .
}
\] 
As in \ref{def::dLieRinehart}, we denote by:
\[
\DB{-}{-}^\Omega_r \overset{}{:=} pr_{\Sigma\Omega_A\otimes A} \circ \DB{-}{-}^\Omega	\quad \mbox{ and }\quad 
\DB{-}{-}^\Omega_l \overset{}{:=} pr_{\Sigma A\otimes \Omega_A} \circ \DB{-}{-}^\Omega
\]
the projections of $\DB{-}{-}^\Omega$ to $\Sigma\Omega_A\otimes A$ and $\Sigma A\otimes \Omega_A$. By the derivation property of $\DB{-}{-}$ and proposition \ref{prop::sDer_representable}, we naturally extend the double bracket to two $A$-bimodule morphisms:
\begin{itemize}
	\item the morphism
	\[
	\rho_l^\Omega : \Sigma \Omega_A \otimes \Sigma A \longrightarrow \Sigma A\otimes A
	\] 
	with, for the left term, the $A$-bimodule structure induced by that of $\Omega_A$ and, for the right term, the internal structure;
	\item the morphism
	\[
	\rho_r^\Omega : \Sigma A \otimes \Sigma \Omega_A \longrightarrow \Sigma A\otimes A
	\] 
	with for the left term, the induced $A$-bimodule structure by that of $\Omega_A$ and the external structure for the right term.
\end{itemize}
By definition, the following diagrams commute:
\[
\xymatrix@C=2cm{
	\Sigma A\otimes \Sigma A \ar[r]^{\DB{-}{-}} \ar[d]_{\Sigma d \otimes \Sigma A}& \Sigma A\otimes A  \\
	\Sigma \Omega_A \otimes \Sigma A \ar@/_1pc/[ru]_{\rho^\Omega_l} &
}
\quad \mbox{ and }\quad 
\xymatrix@C=2cm{
	\Sigma A\otimes \Sigma A \ar[r]^{\DB{-}{-}} \ar[d]_{\Sigma A \otimes \Sigma d} & \Sigma A\otimes A  \\
	\Sigma A \otimes \Sigma \Omega_A \ar@/_1pc/[ru]_{\rho^\Omega_r} .& 
}
\]
As $\DB{-}{-}$ is antisymmetric, we have the following anticommutative diagram:
\[
\xymatrix{
	\Sigma \Omega_A \otimes \Sigma A \ar[r]^{\rho^\Omega_l} \ar[d]_{\tau_{\Sigma \Omega_A,\Sigma A}}\ar@{}[rd]|{\circleddash}
	& \Sigma A\otimes A \ar[d]^{\Sigma \tau_{A,A}}\\
	\Sigma A \otimes \Sigma \Omega_A	\ar[r]_{\rho^\Omega_r} &\Sigma A\otimes A.
}%
\]
We simplify the notation $\rho^\Omega:= \rho ^\Omega_l$ and, in the next proposition, we prove that the morphisms $\rho^\Omega$ and $\DB{-}{-}^\Omega$ provide the $A$-bimodule $\Omega_A$ with a $\Sigma$-double Lie-Rinehart algebra structure. 

\begin{rem}
	Van den Bergh gives a similar construction in  \cite[Prop. A.2.1]{VdB08} but with a weight shifting. Here, we generalize his result without the shifting: Section  \ref{subsect::decalage} allows us to recover Van den Bergh's result.
\end{rem}

\begin{theorem}\label{thm::koszul}
	Let $A$ be a $\Sigma$-double Poisson algebra in an additive symmetric monoidal category  $(\C,\otimes,\tau)$ with coequalizer. The morphisms $\rho^{\Omega}$ and $\DB{-}{-}^{\Omega}$ give  the $A$-bimodule $\Omega_A$ a $\Sigma$-double Lie-Rinehart algebra structure. 	Equivalently, (cf. proposition \ref{prop::eq_LieRinehart_DPlin}), the free  $A$-algebra $T_A\Omega_A$ is a linear $\Sigma$-double Poisson algebra.
\end{theorem}
\begin{proof}
	Let $(A,\DB{-}{-})$ be a $\Sigma$-double Poisson algebra, we note $\Omega:=\Omega_A$ with $_A\mu^\Omega$ and $\mu_A^\Omega$ the morphisms which define the $A$-bimodule structure of $\Omega$. We'll show that the morphisms $\rho^{\Omega}$ and $\DB{-}{-}^{\Omega}$ give a $\Sigma$-double Lie-Rinehart algebra structure on $\Omega$.	The double bracket $\DB{-}{-}^\Omega$ is antisymmetric: in fact, it is defined by the following commutative diagram 
	\[
	\xymatrix@C=2cm{
		\Sigma A\otimes \Sigma A \ar[r]^{\DB{-}{-}} \ar[d]_{d\otimes d}& \Sigma A \otimes A \ar[d]^{d\otimes A +A\otimes d} \\
		\Sigma \Omega \otimes \Sigma \Omega \ar[r]_{\DB{-}{-}^\Omega} & \Sigma\big( \Omega\otimes A \oplus A\otimes \Omega \big)
	}
	\]
	then $\DB{-}{-}^{\Omega}$ satisfies the antisymmetry condition \eqref{def::DLR_antisym} of definition \ref{def::dLieRinehart} by the antisymmetry of the double bracket of $A$. 
	
	We'll show that $\rho^\Omega$ and $\DB{-}{-}^\Omega$ satisfy the derivation condition  \eqref{def::DRL_compatible1} of definition \ref{def::dLieRinehart}. 
	We note $A_1=A=A_2$; we have the diagram of figure (\ref{fig::premier_diag_comp}).
	\begin{figure}[h]	
		\begin{eqnarray*}
			&\xymatrix@R=0.75cm@C=2cm{
				& \Sigma \Omega \otimes \Sigma A 
				\ar[rr]|{\rho^\Omega} 
				\ar[dd]|(0.2){\Sigma\Omega\otimes \Sigma d }
				&& \Sigma A\otimes A 
				\ar[dd]|(0.55){\Sigma d\otimes A} \\
				\Sigma \Omega\otimes \Sigma A_1\otimes A_2 
				\ar[ru]|{\Sigma\Omega\otimes\Sigma \mu}
				\ar[rr]|{(\rho^\Omega\otimes A)+(A\otimes\rho^\Omega)(\tau_{\Sigma\Omega\otimes\Sigma,A_1}\otimes A_2)}
				\ar[dd]|(0.35){\Sigma\Omega\otimes \Sigma d\otimes A + \Sigma \Omega \otimes \Sigma A\otimes  d)}
				&&{\substack{\Sigma A\otimes A\otimes A_2 \\  \quad \oplus A_1\otimes \Sigma A\otimes A}} 
				\ar[dd]|(0.3){(\Sigma d\otimes A\otimes A, d\otimes\Sigma A\otimes A+A_1\otimes \Sigma d\otimes A)}
				\ar[ru]|(0.45){(\Sigma A\otimes\mu,(\Sigma\mu\otimes A)(\tau_{A_1,\Sigma}\otimes A^{\otimes 2})}
				& \\
				& \Sigma \Omega \otimes \Sigma\Omega
				\ar[rr]_(0.35){\DB{-}{-}^{\Omega}_l}|(0.51)\hole 
				&& \Sigma \Omega\otimes A \\
				{\substack{\Sigma \Omega \otimes \Sigma \Omega \otimes A_2 \\ \quad \oplus \Sigma\Omega\otimes \Sigma A_1 \otimes \Omega }} 
				\ar[rr]^{\big(\DB{-}{-}_l^\Omega\otimes A_2+\Sigma\Omega\otimes \rho^{\Omega}}_{\quad +(A_1\otimes\DB{-}{-}_l^\Omega)(\tau_{\Sigma\Omega\otimes\Sigma,A_1}\otimes \Omega)\big)}
				\ar[ru]|(0.6){(\Sigma\Omega\otimes \Sigma \mu_A^\Omega,\Sigma\Omega\otimes \Sigma _A\mu^\Omega)\qquad\quad}
				&& {\substack{\Sigma\Omega\otimes A\otimes A_2 \\ \quad \oplus \Omega \otimes \Sigma A\otimes A \\ \quad \oplus A_1\otimes\Sigma\Omega\otimes A}}
				\ar[ru]_{\Phi_l}
				& 
			}~~\\
			&~~\\
			&\mbox{with}\qquad \Phi_l:= \big(\Sigma\Omega\otimes \mu,(\Sigma\mu_A^\Omega\otimes A)(\tau_{\Omega,\Sigma}\otimes A^{\otimes 2}),(\Sigma_A\mu^\Omega\otimes A)(\tau_{A,\Sigma}\otimes \Omega\otimes A) \big).	
		\end{eqnarray*}
		\caption{First diagram of compatibility \eqref{def::DRL_compatible1}}
		\label{fig::premier_diag_comp}
	\end{figure}
	The front and back faces commute by definition of $\DB{-}{-}^\Omega_l$. The left and right faces commute because $d$ is a derivation and the top face commutes because $\rho^{\Omega}$ is a derivation. Then the bottom face commutes. 
	
	Similarly, we have the commutative diagram of figure \eqref{fig::second_diag_comp}.
	\begin{figure}[h]	
		\begin{eqnarray*}
			&\xymatrix@R=0.75cm@C=2cm{
				& \Sigma \Omega \otimes \Sigma A 
				\ar[rr]|{\rho^\Omega} 
				\ar[dd]|(0.2){\Sigma\Omega\otimes \Sigma d }
				&& \Sigma A\otimes A 
				\ar[dd]|(0.55){\Sigma A\otimes d} 
				\\
				\Sigma \Omega\otimes \Sigma A_1\otimes A_2 
				\ar[ru]|{\Sigma\Omega\otimes\Sigma \mu}
				\ar[rr]|{(\rho^\Omega\otimes A)+(A\otimes\rho^\Omega)(\tau_{\Sigma\Omega\otimes\Sigma,A_1}\otimes A_2)}
				\ar[dd]|(0.35){\Sigma\Omega\otimes \Sigma d\otimes A + \Sigma \Omega \otimes \Sigma A\otimes  d)}
				&&{\substack{\Sigma A\otimes A\otimes A_2 \\  \quad \oplus A_1\otimes \Sigma A\otimes A}} 
				\ar[dd]|(0.3){(\Sigma A\otimes d\otimes A+\Sigma A\otimes A\otimes d, A\otimes\Sigma A\otimes d)}
				\ar[ru]|(0.45){(\Sigma A\otimes\mu,(\Sigma\mu\otimes A)(\tau_{A_1,\Sigma}\otimes A^{\otimes 2})}
				& \\
				& \Sigma \Omega \otimes \Sigma\Omega
				\ar[rr]_(0.35){\DB{-}{-}^{\Omega}_r}|(0.51)\hole 
				&& \Sigma A\otimes \Omega\\
				{\substack{\Sigma \Omega \otimes \Sigma \Omega \otimes A_2 \\ \quad \oplus \Sigma\Omega\otimes \Sigma A_1 \otimes \Omega }} 
				\ar[rr]^{\big(\DB{-}{-}_r^\Omega\otimes A_2, \rho^{\Omega}\otimes \Omega}_{\quad + (A_1\otimes\DB{-}{-}_r^\Omega)(\tau_{\Sigma\Omega\otimes\Sigma,A_1}\otimes \Omega)\big)}
				\ar[ru]|(0.6){(\Sigma\Omega\otimes \Sigma \mu_A^\Omega,\Sigma\Omega\otimes \Sigma _A\mu^\Omega)\qquad\quad}
				&& {\substack{\Sigma A\otimes \Omega \otimes A_2  \\ \quad \oplus A_1\otimes\Sigma A\otimes \Omega  \\ \quad \oplus  \Sigma A\otimes A \otimes \Omega}}
				\ar[ru]_{\Phi_r}
				& 
			}~~\\
			&~~\\
			&\mbox{with}\qquad \Phi_r:= (\Sigma A\otimes \mu_A^\Omega,(\Sigma\mu\otimes \Omega)(\tau_{A,\Sigma}\otimes A\otimes \Omega),\Sigma A\otimes _A\mu^\Omega ).	
		\end{eqnarray*}
		\caption{Second diagram of compatibility \eqref{def::DRL_compatible1}}
		\label{fig::second_diag_comp}
	\end{figure}	
	Then, the morphisms $\rho^\Omega$ and $\DB{-}{-}^\Omega$ satisfy the compatibility condition \eqref{def::DRL_compatible1} of  definition \ref{def::dLieRinehart}. We show that $\rho^\Omega$ and $\DB{-}{-}^\Omega$ satisfy the condition \eqref{def::DRL_compatible2}. We call $\Psi_L$ the morphism $\DDBL{-}{-}{-}= \big(\DB{-}{-}\otimes A\big)\big(\Sigma A\otimes \DB{-}{-}\big)$. We have the following diagram with  commuting vertical faces:
	\[
	\xymatrix@C=0.77cm{
		& \big(\Sigma A\big)^{\otimes 3} \ar@/_2pc/[ldd]|{\Psi_L} \ar[d]|{\Sigma d\otimes \Sigma A \otimes \Sigma d}
		&& \big(\Sigma A\big)^{\otimes 3} \ar[rr]^{\tau_{\Sigma A,\Sigma A\otimes \Sigma A}}  \ar[ll]_{\tau_{\Sigma A\otimes \Sigma A, \Sigma A}} 
		\ar@/_2pc/[ldd]|{\Psi_L} \ar[d]|{\Sigma A \otimes\Sigma d\otimes  \Sigma d}
		&& \big(\Sigma A\big)^{\otimes 3} \ar@/_2pc/[ldd]|{\Psi_L} \ar[d]|{(\Sigma d)^{\otimes 2}\otimes \Sigma A}\\
		& \Sigma\Omega\otimes \Sigma A\otimes \Sigma\Omega 
		\ar[ld]|{(\rho^\Omega_l\otimes A)(\Sigma \Omega\otimes \rho^\Omega_r)}
		&& \Sigma A\otimes \Sigma\Omega \otimes \Sigma\Omega 
		\ar[ld]|(0.4){(\rho^\Omega_r\otimes A)(\Sigma A\otimes \DB{-}{-}_l^\Omega)}
		\ar[rr]^(0.4){\tau_{\Sigma A,\Sigma \Omega \otimes \Sigma \Omega }}|(0.57)\hole 
		\ar[ll]_(0.6){\tau_{\Sigma A \otimes \Sigma \Omega , \Sigma \Omega }}|(0.43)\hole
		&&\big( \Sigma\Omega \big)^{\otimes 2} \otimes \Sigma A 
		\ar[ld]|{(\rho^\Omega_l\otimes A)(\Sigma \Omega\otimes \rho^\Omega_l)} \\
		\Sigma \big( A\big)^{\otimes 3}\ar[rr]_{\Sigma\tau_{ A, A\otimes A}} 
		&& \Sigma\big(A\big)^{\otimes 3}  
		&& \Sigma\big(A\big)^{\otimes 3} \ar[ll]^{\Sigma\tau_{A\otimes  A,  A}} &.
	}
	\]
	Then, as $\DB{-}{-}$ satisfies the double Jacobi identity, the term 
	\begin{eqnarray*}
		&&(\rho^\Omega_r\otimes A)(\Sigma A\otimes \DB{-}{-}_l^\Omega) \\
		&&\qquad +\tau_{\Sigma A\otimes \Sigma A, \Sigma A}(\rho^\Omega_l\otimes A)(\Sigma \Omega\otimes \rho^\Omega_l)\tau_{\Sigma A,\Sigma \Omega \otimes \Sigma \Omega }  \\
		&&\qquad + \tau_{\Sigma A,\Sigma A\otimes \Sigma A}(\rho^\Omega_l\otimes A)(\Sigma \Omega\otimes \rho^\Omega_r)\tau_{\Sigma A \otimes \Sigma \Omega , \Sigma \Omega }
	\end{eqnarray*}
	is equal to zero, hence the $\DB{-}{-}^\Omega$ and $\rho^\Omega$ satisfy the condition \eqref{def::DRL_compatible2} of definition \ref{def::dLieRinehart}. We have the following diagram with commuting vertical faces:
	\[
	\xymatrix@C=0.78cm{
		& \big(\Sigma A\big)^{\otimes 3} \ar@/_2pc/[ldd]|{\Psi_L}
		\ar[d]|{(\Sigma d)^{\otimes 3}}
		&& \big(\Sigma A\big)^{\otimes 3} \ar[rr]^{\tau_{\Sigma A,\Sigma A\otimes \Sigma A}}  \ar[ll]_{\tau_{\Sigma A\otimes \Sigma A, \Sigma A}} 
		\ar@/_2pc/[ldd]|{\Psi_L} \ar[d]|{(\Sigma d)^{\otimes 3}}
		&& \big(\Sigma A\big)^{\otimes 3} \ar@/_2pc/[ldd]|{\Psi_L} \ar[d]|{(\Sigma d)^{\otimes 3}}\\
		& \big(\Sigma\Omega\big)^{\otimes 3}
		\ar[ld]|{(\DB{-}{-}^\Omega_r\otimes A)(\Sigma \Omega \otimes \DB{-}{-}^\Omega_l)}
		&& \big(\Sigma\Omega\big)^{\otimes 3}
		\ar[ld]|(0.4){(\DB{-}{-}_l^\Omega\otimes A)(\Sigma \Omega\otimes \DB{-}{-}_l^\Omega)}
		\ar[rr]^(0.37){\tau_{\Sigma \Omega,\Sigma \Omega \otimes \Sigma \Omega }}|(0.53)\hole 
		\ar[ll]_(0.63){\tau_{\Sigma \Omega \otimes \Sigma \Omega , \Sigma \Omega }}|(0.46)\hole
		&&\big(\Sigma\Omega\big)^{\otimes 3}
		\ar[ld]|{~~~~(\rho^\Omega_l\otimes A)(\Sigma \Omega\otimes\DB{-}{-}^\Omega_r)}  ~;\\
		\Sigma A\otimes \Omega \otimes A \ar[rr]_{\Sigma\tau_{ A, \Omega\otimes  A} } 
		&& \Sigma\Omega\otimes A^{\otimes 2}  
		&& \Sigma A^{\otimes 2} \otimes \Omega \ar[ll]^{\Sigma\tau_{ A\otimes  A, \Omega}} &
	}
	\]
	as $\DB{-}{-}$ satisfies the double Jacobi identity, the term
	\begin{eqnarray*}
		&&(\DB{-}{-}_l^\Omega\otimes A)(\Sigma \Omega\otimes \DB{-}{-}_l^\Omega) \\
		&&\qquad +\Sigma\tau_{ A, \Omega\otimes  A}(\DB{-}{-}^\Omega_r\otimes A)(\Sigma \Omega \otimes \DB{-}{-}^\Omega_l)\tau_{\Sigma \Omega \otimes \Sigma \Omega , \Sigma \Omega }  \\
		&&\qquad + \Sigma\tau_{ A\otimes  A, \Omega}(\rho^\Omega_l\otimes A)(\Sigma \Omega\otimes\DB{-}{-}^\Omega_r)\tau_{\Sigma \Omega,\Sigma \Omega \otimes \Sigma \Omega }
	\end{eqnarray*}
	is equal to zero. By invariance under the $\Z/3\Z$-action of the double jacobiator, the double bracket $\DB{-}{-}^\Omega$ satisfies the double Jacobi identity.
\end{proof}

With the shifting property, we recove the original construction of the Koszul double bracket of Van den Bergh. In fact, in \cite[Ann. A]{VdB08}, Van den Bergh constructs the Koszul double bracket as a Gerstenhaber double bracket, i.e. a Poisson double bracket of degree $-1$, on $T_A(\Omega_A[1])$. The shifting property \ref{prop::equivalence_decalage}, applied on our Koszul double bracket construction, returns the original Koszul double bracket of Van den Bergh.

\subsection{Example : the Schouten-Nijenhuis double bracket}\label{sect::Schouten}
For this example, we take $\C=\DGA_k$ with $k$ a field, the category of differential $\Z$-graded associative algebras and we consider $A\in \Ob \DGA_k$.
We start by the definition of the $A$-bimodule of biderivations de $A$; biderivations play the role of derivations in noncommutative geometry. 
\begin{definition}[Biderivation]\label{def::biderivation}
	We define $\DDer(A)$, \emph{the external $A$-bimodule of biderivations of $A$}, by
	\[
		\DDer(A):= \Der(A,A\otimes A)
	\]
	where $A\otimes A$ is equipped with its external $A$-bimodule structure. 
\end{definition}

We recall the definition of the Schouten Nijenhuis $0$-double Poisson bracket following Van den Bergh  in \cite[Sect. 3.2]{VdB08}. We consider $A$, a \emph{finitely generated} differential graded algebra:  this implies that the $A$-bimodule $\Omega_A$ is finitely generated. The morphisms $\Phi, \Psi : \DDer(A)^{\otimes 2}\otimes A \rightarrow A^{\otimes 3}$ defined by
\begin{align*}
	\Phi:= &(A\otimes \tau_{A,A})\Big((ev\otimes A)(\DDer(A)\otimes ev) \\
	&\qquad- (A\otimes ev)(\tau_{\DDer(A),A}\otimes A)(\DDer(A)\otimes ev)(\tau_{\DDer(A),\DDer(A)}	\otimes A)\Big)  \\
	\Psi:= &(\tau_{A,A}\otimes A)\Big((A\otimes ev)(\tau_{\DDer(A),A}\otimes A)(\DDer(A)\otimes ev) \\
	& \qquad - (ev\otimes A)(\DDer(A)\otimes ev)(\tau_{\DDer(A),\DDer(A)}	\otimes A) \Big)
\end{align*}

give us, by adjunction, the morphisms
\begin{eqnarray*}
	\Phi^*:= \DB{-}{-}^{SN}_l :  \DDer(A)\otimes \DDer(A) \rightarrow \Hom(A,  A^{\otimes 3}), \\
	\Psi^*:= \DB{-}{-}^{SN}_r : \DDer(A)\otimes \DDer(A) \rightarrow \Hom(A,  A^{\otimes 3}).
\end{eqnarray*}
\begin{prop}\textup{\cite[Prop. 3.2.1]{VdB08}} 
	The morphism $\DB{-}{-}^{SN}_l $ (resp. $\DB{-}{-}^{SN}_r$) factorizes through $\DDer(A)\otimes A$ (resp. $A\otimes \DDer(A)$).
\end{prop}
\begin{proof}
	The $A$-bimodule $\Omega_A$ is finitely generated: this implies the chain complex isomorphism $
	\Der(A,A^{\otimes 3})\cong \Hom_{A^e}(\Omega_A, A\otimes A) \otimes A$.
\end{proof}


\begin{thm}\textup{\cite[Thm. 3.2.2]{VdB08}}~\label{thm::dcrochet_schouten} 
	The morphisms 
	\begin{eqnarray*}
		\DB{-}{-}^{SN}_l +\DB{-}{-}^{SN}_r : \DDer(A)\otimes \DDer(A) \rightarrow \DDer(A)\otimes A\oplus A\otimes \DDer(A)	\\
		(ev, -\tau_{A,A}\circ ev\circ\tau_{A,\DDer(A)}):\DDer(A)\otimes A\oplus A\otimes \DDer(A) \rightarrow A\otimes A
	\end{eqnarray*}
	give us an $A$-linear $0$-double Poisson algebra structure on the free $A$-algebra $T_A\DDer(A)$.
\end{thm}
\begin{proof}
	We take Van den Bergh's theorem \cite[Th. 3.2.2]{VdB08} and apply \ref{prop::equivalence_decalage}.
\end{proof}

\begin{rem}[Relation between Schouten Nijenhuis and Koszul double brackets]
	Let $A$ be a finitely-generated differential graded algebra with a double Poisson bracket. The $A$-bimodule $\Omega_A$ is a double Lie-Rinehart algebra for the Koszul structure $(\rho^\Omega,\DB{-}{-}^\Omega)$ (see section \ref{subsect::dcrochet_Koszul}). 
	The category $\Ch_k$ is closed, so the anchor $\rho^\Omega$ corresponds to a morphism $(\rho^\Omega)^*: \Omega \rightarrow \Hom(A,A\otimes A)$ which factorizes through $\DDer(A)$.
	As remarked in \ref{rem::DLieRin_closed_cat}, in this case,  the compatibility condition \eqref{def::DRL_compatible2} can be expressed using the Schouten-Nijenhuis double bracket introduced in \ref{thm::dcrochet_schouten}, as:
	\[
	(\rho^\Omega)^*\big( \DB{-}{-}^M\big) = \DB{(\rho^\Omega)^*}{(\rho^\Omega)^*}^{SN}.
	\]
	Then, the morphism $(\rho^\Omega)^*$ is an $A$-double Lie Rinehart morphism between $\Omega_A$ and $\DDer(A)$.
\end{rem}


\nocite{Ler17}
\bibliographystyle{alpha}
\bibliography{biblio_these}

\end{document}

%% file: commandes.tex
\theoremstyle{plain}

\newtheorem{thm}{Theorem}[subsection]
\newtheorem*{thm*}{Theorem}

\declaretheoremstyle[
headfont=\normalfont\bfseries,
notebraces={(}{)},
bodyfont=\itshape,
headpunct={}
]{mythmstyle}

\declaretheorem[
style=mythmstyle,
numberlike=thm,
name=Theorem
]{theorem}

\declaretheorem[
style=mythmstyle,
numbered=no,
name=Theorem
]{theorem*}

\declaretheorem[
style=mythmstyle,
numberlike=thm,
name=Lemma
]{lem}

\declaretheorem[
style=mythmstyle,
numberlike=thm,
name=Proposition
]{prop}

\declaretheorem[
style=mythmstyle,
numbered=no,
name=Proposition
]{prop*}

\usepackage{xparse}

\declaretheoremstyle[
headfont=\normalfont\bfseries,
notebraces={(}{)},
bodyfont=\normalfont,
headpunct={}
]{mystyle}

\declaretheorem[
style=mystyle,
numberlike=thm,
name=Definition
]{definitionnoindex}

\declaretheorem[
style=mystyle,
numbered=no,
name=Definition
]{definitionnoindex*}

\usepackage{xparse}
\NewDocumentEnvironment{definition}{oo}
	{\IfNoValueTF{#1}
		{\begin{definitionnoindex}}	
		{\IfNoValueTF{#2}
			{\begin{definitionnoindex}[#1]\index{#1}}
			{\begin{definitionnoindex}[#1]\index{#2}}
		}
	}
	{\end{definitionnoindex}}%

\declaretheorem[
style=mystyle,
numberlike=thm,
name=Definition/Proposition.
]{definitionpropnoindex}

\NewDocumentEnvironment{definitionprop}{oo}
{\IfNoValueTF{#1}
	{\begin{definitionpropnoindex}}	
		{\IfNoValueTF{#2}
			{\begin{definitionpropnoindex}[#1]\index{#1}}
				{\begin{definitionpropnoindex}[#1]\index{#2}}
				}
			}
			{\end{definitionpropnoindex}}%

\declaretheorem[
style=mystyle,
numberlike=thm,
name=Remark
]{rem}

\declaretheorem[
style=mystyle,
numberlike=thm,
name=Example
]{exmp}

\theoremstyle{definition}

\newtheorem{exmps}[thm]{Examples}
\newtheorem{nota}[thm]{Notation(s)}

\RequirePackage[normalem]{ulem} 
\RequirePackage{color}\definecolor{RED}{rgb}{1,0,0}\definecolor{BLUE}{rgb}{0,0,1} 

\newcommand{\Z}{\mathbb{Z}}
\newcommand{\N}{\mathbb{N}}

\newcommand{\C}{\mathtt{C}}

\newcommand{\g}{\mathfrak{g}}
\renewcommand{\phi}{\varphi}

\newcommand{\DB}[2]{\{\hspace{-3pt}\{#1,#2\}\hspace{-3pt}\}}

\newcommand{\DDBL}[3]{\big\{\hspace{-4pt}\big\{#1,\DB{#2}{#3}\big\}\hspace{-4pt}\big\}_L}

\newcommand{\TB}[3]{\{\hspace{-3pt}\{#1,#2,#3\}\hspace{-3pt}\}}

\newcommand{\Ob}{\mathrm{Ob}\,}

\newcommand{\Ch}{\mathtt{Ch}}

\newcommand{\DGA}{\mathtt{DGA}}

\newcommand{\DPoiss}{\mathtt{DPoiss}}
\newcommand{\catDLie}{\mathtt{DLie}}

\newcommand{\GL}{\mathrm{GL}}

\newcommand{\Der}{\mathrm{Der}}

\newcommand{\End}{\mathrm{End}\,}
\newcommand{\EEnd}{\mathbb{E}\mathrm{nd}}

\newcommand{\Hom}{\mathrm{Hom}}
\newcommand{\sHom}{\mathrm{\underline{hom}}}
\newcommand{\DDer}{\mathbb{D}\mathrm{er}}

\newcommand{\commu}{\circlearrowleft}

